\documentclass[11pt,reqno]{amsart}
\usepackage{amsfonts,latexsym,amsmath, amssymb}
\usepackage{mathrsfs}
\usepackage{bm}
\usepackage{color}
\usepackage{amsmath}
\usepackage{cite}
\usepackage[utf8]{inputenc}
\usepackage{graphicx,enumerate}

\newcommand{\bea}{\begin{eqnarray}}
\newcommand{\eea}{\end{eqnarray}}
\def\beaa{\begin{eqnarray*}}
\def\eeaa{\end{eqnarray*}}
\def\ba{\begin{array}}
\def\ea{\end{array}}
\def\be#1{\begin{equation} \label{#1}}
\def \eeq{\end{equation}}

\newcommand{\Le}{{\rm{{Le}}}}

\def\be{{\beta}}

\def\varep{\varepsilon}

\def\R{{\mathbb{R}}}
\def\C{{\mathbb{C}}}

\def\N{{\mathbb N}}

\def\Re{\mathop{\rm Re}}
\def\Im{\mathop{\rm Im}}

\newtheorem{theorem}{Theorem}[section]
\newtheorem{lemma}[theorem]{Lemma}
\newtheorem{proposition}[theorem]{Proposition}

\newtheorem{definition}[theorem]{Definition}
\newtheorem{remark}[theorem]{Remark}

\numberwithin{equation}{section}

\numberwithin{equation}{section}

\allowdisplaybreaks
\begin{document}

\title[Stability analysis and Hopf bifurcation]{Stability analysis and Hopf bifurcation at high Lewis number in a combustion model with free interface}
\author{Claude-Michel Brauner}
\address{School of Mathematical Sciences, University of Science and Technology of China, Hefei 230026 (China), and
Institut de Math\'ematiques de Bordeaux, Universit\'e de Bordeaux, 33405 Talence Cedex (France).}

\author{Luca Lorenzi}
\address{Dipartimento di Scienze Matematiche, Fisiche e Informatiche, Plesso di Matematica e Informati\-ca, Universit\`a di Parma, Parco Area delle Scienze 53/A, I-43124 Parma (Italy)}

\author{Mingmin Zhang}\address{School of Mathematical Sciences, University of Science and Technology of China, Hefei 230026 (China).}
\email{claude-michel.brauner@u-bordeaux.fr}
\email{luca.lorenzi@unipr.it}
\email{lang925@mail.ustc.edu.cn}

\keywords{Free interface problem; traveling wave solutions; fully nonlinear parabolic systems; stability; Hopf bifurcation; combustion}
\subjclass[2010]{Primary: 35R35; Secondary: 35K55, 35B35, 80A25}

\begin{abstract}
In this paper we analyze the stability of the traveling wave solution for an ignition-temperature, first-order reaction model of thermo-diffusive combustion, in the case of high Lewis numbers ($\Le >1$). The system of two parabolic PDEs is characterized by a free interface at which ignition temperature $\Theta_i$ is reached. We turn the model to a fully nonlinear problem in a fixed domain. When the Lewis number is large, we define a bifurcation parameter $m=\Theta_i/(1-\Theta_i)$ and a perturbation parameter $\varep= 1/\Le$. The main result is the existence of a critical value $m^c(\varep)$ close to $m^c=6$ at which Hopf bifurcation holds for $\varep$ small enough. Proofs combine spectral analysis and non-standard application of Hurwitz Theorem with asymptotics as $\varep \to 0$.
\end{abstract}
\maketitle
\section{introduction}\label{intro}

This paper is devoted to the stability analysis of a unique (up to translation) traveling wave solution  to a thermo-diffusive model of flame propagation with stepwise temperature kinetics and first-order reaction (see \cite{BGKS15}) at high Lewis numbers, namely $\Le>1$. The problem reads in one spatial dimension:
\begin{eqnarray}\label{problem-1}
\left\{
\begin{array}{l}
\displaystyle\frac{\partial\Theta}{\partial t}=\frac{\partial^2\Theta}{\partial x^2}+W(\Theta,\Phi), \\[2mm]
\displaystyle\frac{\partial\Phi}{\partial t}={\Le}^{-1}\frac{\partial^2\Phi}{\partial x^2}-W(\Theta,\Phi).
\end{array}
\right.
\end{eqnarray}
Here, $\Theta$ and $\Phi$ are appropriately normalized temperature and concentration of deficient reactant,
$x\in \R$ denotes the spatial coordinate, $t>0$ the time. The nonlinear term $W(\Theta,\Phi)$ is a scaled reaction rate given by (see \cite[Section 2, formula (3)]{BGKS15}):
\begin{eqnarray}\label{problem-2}
W(\Theta,\Phi)=
\left\{
\begin{array}{lllll}
A\Phi,  & \mbox{if} & \Theta\ge \Theta_i, \\[2mm]
0, & \mbox{if} &\Theta<\Theta_i.
\end{array}
\right.
\end{eqnarray}
In \eqref{problem-2}, $0<\Theta_i<1$ is the reduced ignition temperature, $A>0$ is a normalized factor depending on $\Theta_i$ and $\Le$, to be determined hereafter for the purpose of ensuring that the speed of traveling wave is set at unity. Moreover, the following boundary conditions hold at $\pm \infty$:
\begin{eqnarray}
\label{boundary condition-12}
\begin{matrix}
\Theta(t, -\infty)=1, & & \Theta(t,\infty)=0,\\
\Phi(t, -\infty)=0, & & \Phi(t,\infty)=1.
\end{matrix}
\end{eqnarray}

In this first-order stepwise kinetics model, $\Phi$ does not vanish except as $t$ tends to $-\infty$. Thus, problem \eqref{problem-1}-\eqref{boundary condition-12} belongs to the class of parabolic Partial Differential Equations with discontinuous nonlinearities. Models in combustion theory and other fields (see, e.g. \cite[Section 1]{AF82}) involving discontinuous reaction terms have been used by physicists and engineers for long because of their manageability; as a result, elliptic and parabolic PDEs with discontinuous nonlinearities, and related Free Boundary Problems, have received a close attention from the mathematical community (see \cite[Section 1]{ABLZ18} and references therein). We quote in particular the paper \cite{C80}, by K.-C. Chang, which contains a systematical study of elliptic PDEs with discontinuous nonlinearities (DNDE).
	
In this paper, we consider the case of a free \textsl{ignition interface} $g(t)$ defined by
\begin{equation}
\label{ignition}
\Theta(t,g(t))=\Theta_i,
\end{equation}
such that $\Theta(t,x)>\Theta_i$ for $x>g(t)$ and  $\Theta(t,x)<\Theta_i$ for $x<g(t)$. Formula (\ref{ignition}) means that the ignition temperature $\Theta_i$ is reached at the ignition interface which defines the flame front. We point out that,  in contrast to conventional Arrhenius kinetics where the reaction zone is infinitely thin, the reaction zone for stepwise temperature kinetics is of order unity (thick flame). It is also interesting to compare the first-order stepwise kinetics with the zero-order kinetics model (see \cite{ABLZ18,BGKS15,BGZ16}): in the zero-order kinetics, $\Phi(t,x)$ vanishes at a \textit{trailing interface} and does not appear explicitly in the nonlinear term (see \cite[Section 2,  formula (4)]{BGKS15}).

According to \eqref{ignition}, the system for $\pmb{X}= (\Theta,\Phi)$ reads as follows, for $ t>0$ and $x \in \R, x \neq g(t)$:
\begin{eqnarray}
\label{system-1}
&&\left\{\begin{aligned}
&\frac{\partial\Theta}{\partial t}=\frac{\partial^2\Theta}{\partial x^2}+A\Phi,&x<g(t),\\
&\frac{\partial\Phi}{\partial t}={\rm \Le}^{-1}\frac{\partial^2\Phi}{\partial x^2}-A\Phi,\quad &x<g(t),
\end{aligned}\right.\\[2mm]
\label{system-2}
&&\left\{\begin{aligned}
&\frac{\partial\Theta}{\partial t}=\frac{\partial^2\Theta}{\partial x^2},&x>g(t),\\
&\frac{\partial\Phi}{\partial t}={\rm \Le}^{-1}\frac{\partial^2\Phi}{\partial x^2},\quad &x>g(t).
\end{aligned}\right.
\end{eqnarray}
At the free interface $x=g(t)$, the following continuity conditions hold:
\begin{equation}\label{system 1-2}
[\Theta]=[\Phi]=0, \qquad\;\, \bigg [\frac{\partial\Theta}{\partial x}\bigg ]=\bigg [\frac{\partial\Phi}{\partial x}\bigg ]=0,
\end{equation}
where we denote by $[f]$ the jump of a function $f$ at a point $x_0$, i.e., the difference $f(x_0^+)-f(x_0^-)$.

The system above admits a unique (up to translation) traveling wave solution $\pmb U=(\Theta^0,\Phi^0)$ which propagates with constant positive velocity $V$. In the moving frame coordinate $z=x-Vt$, by choosing
\begin{equation}
\label{eqn:A}
A=\frac{\Theta_i}{1-\Theta_i}\bigg(1+\frac{\Theta_i}{\Le(1-\Theta_i)}\bigg),
\end{equation}
to have $V=1$ and, hence, $z=x-t$, the traveling wave solution is explicitly given by the following formulae:
\begin{eqnarray*}
&\Theta^0(z)&=
\left\{\begin{aligned}
&1-(1-\Theta_i)e^{\frac{\Theta_i}{1-\Theta_i}z},&  \ \  &z<0,&\\
&\Theta_ie^{-z},&  \ \  &z>0,&
\end{aligned}\right.\\[2mm]
&\Phi^0(z)&=
\left\{\begin{aligned}
&\frac{\Theta_i}{A(1-\Theta_i)}e^{\frac{\Theta_i}{1-\Theta_i}z},&  \ \  &z<0,&\\
&1+\left (\frac{\Theta_i}{A(1-\Theta_i)}-1\right )e^{-\Le z},&  \ \  &z>0.&
\end{aligned}\right.
\end{eqnarray*}

The goal of this paper is the analysis of the stability of the traveling wave solution $\pmb U$ in the case of high Lewis numbers ($\Le>1$). Here, stability refers to orbital stability with asymptotic phase, because of the translation invariance of the traveling wave. It is known (see \cite[Section 3.2]{BGKS15})
that large enough Lewis numbers give rise to \textit{pulsating instabilities}, i.e., oscillatory behavior of the flame. This is
very unlike \textit{cellular instabilities} for relatively small Lewis number ($\Le<1$), that is pattern formation; in the latter case, a paradigm for  the evolution of the disturbed flame front is the Kuramoto-Sivashinsky equation (see \cite{MS79, S80}, and also \cite{BHL13,BHL09, BHL11, BHLS10, BLSX10}).

The paper is organized as follows: In Section \ref{linear operator}, we first transform the free interface problem to a system of parabolic equations on a fixed domain. Then, in the spirit of  \cite{BHL00,Lorenzi02,Lorenzi02-b}, the perturbation $\pmb u$ of the traveling wave $\pmb U$ is split as $\displaystyle\pmb u= s\frac{d\pmb U}{d\xi} +\pmb v$ (``ansatz 1''), in which $s$ is the perturbation of the front $g$. The largest part of the section is devoted to a thorough study of the linearization at $0$ of the elliptic part of the parabolic system in a weighted space $\bm{\mathcal W}$ where its realization $L$ is sectorial (see Subsection \ref{linearized subsect} for further details about the use of a weighted space). Furthermore, we determine the spectrum of $L$ which contains $(-\infty,-\frac{1}{4}]$, a parabola and its interior, the roots of the so-called dispersion relation, and the eigenvalue $0$.
Thereafter, an important point is getting rid of the eigenvalue 0 which, as it has been already stressed, is generated by translation invariance. In Section \ref{sect-3}, we use a spectral projection $P$ as well as ``ansatz 2'' and then derive the fully nonlinear problem (see, e.g. \cite{Lunardi96}) for $\pmb w$:
\begin{equation*}
\frac{\partial\pmb w}{\partial\tau} = (I-P)L\pmb w + {F}(\pmb w).
\end{equation*}

Next, in Sections \ref{stability} and \ref{sect-5} we use the bifurcation parameter $m$ defined by
\begin{eqnarray*}
m:=\frac{\Theta_i}{1-\Theta_i}
\end{eqnarray*}
to investigate the stability of the traveling wave. Simultaneously, as one already noted that pulsating instability is likely to occur at large Lewis number, it is natural to introduce a small perturbation parameter $\varepsilon >0$ (dimensionless diffusion coefficient) defined by $\varepsilon:=\Le^{-1}$, so that \eqref{eqn:A} reads $A =m+\varep m^2$.
The simplest situation arises in the asymptotic case of gasless combustion when $\Le =\infty$, as in \cite{GJ09}. As it is easily seen, as $\varepsilon \to 0$, problem \eqref{system-1}-\eqref{system-2} converges formally to:
\begin{eqnarray}
\label{system-1-limit}
&&\left\{\begin{aligned}
&\frac{\partial\Theta}{\partial t}=\frac{\partial^2\Theta}{\partial x^2}+A\Phi,\quad &x<g(t),\\
&\frac{\partial\Phi}{\partial t}=-A\Phi,&x<g(t),
\end{aligned}\right.\\[2mm]
\label{system-2-limit}
&&\left\{\begin{aligned}
&\frac{\partial\Theta}{\partial t}=\frac{\partial^2\Theta}{\partial x^2},\quad &x>g(t),\\
&\Phi\equiv 1,&x>g(t),
\end{aligned}\right.
\end{eqnarray}
with conditions $[\Theta]= [\Phi] = 0$, $\displaystyle\left [\frac{\partial\Theta}{\partial x}\right ]=0$ at the free interface $x=g(t)$. However,
the limit free interface system \eqref{system-1-limit}-\eqref{system-2-limit} is only partly parabolic.

At the outset, we fix $m$ in Section \ref{stability} and let $\varep$ tend to $0$, which allows to apply the classical Hurwitz Theorem in complex analysis to the \textit{dispersion relation} $D_{\varep}(\lambda,m)$. Our first main result, Theorem \ref{stability theorem TW}, states that, for $2<m<m^c=6$ and $0<\varep<\varep_0(m)$, the traveling wave $\pmb U$ is orbitally stable with asymptotic phase and, for $m>m^c=6$, it is unstable. To give a broad picture, we take advantage of the regular convergence of the point spectrum as $\varep \to 0$.

Section \ref{sect-5} is devoted to the proof of Hopf bifurcation in a neighborhood of the critical value $m^c=6$. The difficulty is twofold: first, the framework is that of a fully nonlinear problem; second, $m$ is not fixed in the sequence of parameterized analytic functions $D_{\varep}(\lambda,m)$ which prevents us from using Hurwitz Theorem directly. The trick is to find a proper approach to combining $m$ with  $\varep$: to this end we construct a sequence of critical values $m^c(\varep)$ such that $m^c(0)=m^c$ and apply Hurwitz Theorem to $D_{\varep}(\lambda,m^c(\varep))$. Proposition \ref{give critical value} and Theorem \ref{Hopf bifurcation theorem} are crucial to prove Hopf bifurcation at $m^c(\varep)$ for $\varep$ small enough. Finally, in three appendices, we collect some formulae and results that we use to prove our main results.

\section{The linearized operator}\label{linear operator}
In this section, we first derive the governing equations for the perturbations of the traveling wave solution. As usual, it is convenient to transform the free interface problem to a system on a fixed domain. More specifically, we use the general method of \cite{BHL00} that converts free interface problems to fully nonlinear problems with
transmission conditions at a fixed interface (see \cite{ABLZ18}). Then, we are going to focus on the linearized system.

\subsection{The system with fixed interface}\label{fixed}
To begin with, we rewrite problem \eqref{system-1}-\eqref{system 1-2} in a new system of coordinates that fixes the position of the ignition interface at the origin:
\begin{eqnarray*}
\tau=t,\ \
\xi=x-g(\tau).
\end{eqnarray*}
Hereafter, we are going to use, whenever it is convenient, the superdot to denote differentiation with respect to time and the prime to denote partial differentiation with respect to the space variable.

Then, the system for ${\pmb X}=(\Theta,\Phi)$ and $g$ reads:
\begin{eqnarray}
\label{perturbation T}
&&\left\{\begin{aligned}
\frac{\partial\Theta}{\partial\tau}-\dot{g}\frac{\partial\Theta}{\partial\xi}=&\frac{\partial^2\Theta}{\partial\xi^2}+A\Phi, \ \ &\xi<0,\\
\frac{\partial\Phi}{\partial\tau}-\dot{g}\frac{\partial\Phi}{\partial\xi}=&\Le^{-1}\frac{\partial^2\Phi}{\partial\xi^2}-A\Phi, \ \ &\xi<0, \end{aligned}\right.\\[1mm]
\label{perturbation W}
&&\left\{\begin{aligned}
\frac{\partial\Theta}{\partial\tau}-\dot{g}\frac{\partial\Theta}{\partial\xi}=&\frac{\partial^2\Theta}{\partial\xi^2}, \ \ &\xi>0,\\
\frac{\partial\Phi}{\partial\tau}-\dot{g}\frac{\partial\Phi}{\partial\xi}=&\Le^{-1}\frac{\partial^2\Phi}{\partial\xi^2}, \ \ &\xi>0. \end{aligned}\right.
\end{eqnarray}
Moreover, $\Theta$, $\Phi$ and their first-order space derivatives are continuous at the fixed interface $\xi=0$, thus
\begin{equation}
\Theta(\cdot,0)=\Theta_i, \qquad\;\,  [\Theta]=[\Phi]=0, \qquad\;\, \bigg [\frac{\partial\Theta}{\partial\xi}\bigg ]=\bigg [\frac{\partial\Phi}{\partial\xi}\bigg ]=0.
\label{interface-Theta-Phi}
\end{equation}
In addition, at $\xi=\pm \infty$, $\Theta$ and $\Phi$ satisfy \eqref{boundary condition-12}.

Next, we introduce the small perturbations $\pmb u=(u_1,u_2)$ and $s$, respectively of the traveling wave $\pmb U$ and of the front $g$, more precisely,
\begin{align*}
&u_1(\tau,\xi)=\Theta(\tau, \xi)-\Theta^0(\xi),\\
&u_2(\tau,\xi)=\Phi(\tau, \xi)-\Phi^0(\xi),\\
&s(\tau)=g(\tau)-\tau.
\end{align*}
It then follows that the perturbations $\pmb u$ and $s$ verify the system
\begin{eqnarray}
\label{u_1}
&&\left\{\begin{aligned}
\frac{\partial u_1}{\partial \tau}&=\frac{\partial^2 u_1}{\partial \xi^2}+\frac{\partial u_1}{\partial \xi}+Au_2+\dot{s}\frac{d\Theta^0}{d\xi}+\dot{s}\frac{\partial u_1}{\partial \xi}, \ &\xi<0,&\\
\frac{\partial u_2}{\partial \tau}&=\Le^{-1}\frac{\partial^2 u_2}{\partial \xi^2}+\frac{\partial u_2}{\partial \xi}-Au_2+\dot{s}\frac{d\Phi^0}{d\xi}+\dot{s}\frac{\partial u_2}{\partial \xi}, \ &\xi<0,&
\end{aligned}\right.\\[1mm]
\label{u_2}
&&\left\{\begin{aligned}
\frac{\partial u_1}{\partial \tau}&=\frac{\partial^2 u_1}{\partial \xi^2}+\frac{\partial u_1}{\partial \xi}+\dot{s}\frac{d\Theta^0}{d\xi}+\dot{s}\frac{\partial u_1}{\partial \xi}, \ &\xi>0,&\\
\frac{\partial u_2}{\partial \tau}&=\Le^{-1}\frac{\partial^2 u_2}{\partial \xi^2}+\frac{\partial u_2}{\partial \xi}+\dot{s}\frac{d\Phi^0}{d\xi}+\dot{s}\frac{\partial u_2}{\partial \xi}, \ &\xi>0,&
\end{aligned}\right.
\end{eqnarray}
and the corresponding interface conditions obtained from \eqref{interface-Theta-Phi} are:
\begin{equation}\label{interface-u}
u_1(\tau,0)=0,\qquad\;\, [u_1]=[u_2]=\bigg [\frac{\partial u_1}{\partial \xi}\bigg ]=\bigg [\frac{\partial u_2}{\partial \xi}\bigg ]=0.
\end{equation}

\subsection{Ansatz 1}
\label{subsect-2.2}
In the spirit of \cite{BHL00,Lorenzi02}, we introduce the following splitting or ansatz:
\begin{equation}\label{ansatz1}
\begin{aligned}
u_1(\tau,\xi)=&s(\tau)\frac{d\Theta^0}{d\xi}(\xi)+v_1(\tau,\xi),\\
u_2(\tau,\xi)=&s(\tau)\frac{d\Phi^0}{d\xi}(\xi)+v_2(\tau,\xi),
\end{aligned}
\end{equation}
in which $v_1$, $v_2$ are new unknown functions. In a more abstract setting, the ansatz reads
\begin{equation*}
\pmb u(\tau,\xi)= s(\tau)\frac{d\pmb U}{d\xi} +\pmb v(\tau,\xi), \qquad\;\, \pmb v=(v_1,v_2).
\end{equation*}
Substituting \eqref{ansatz1} into \eqref{u_1}-\eqref{u_2}, we get the system for $\pmb u$ and $s$:
\begin{eqnarray}
\label{v1}
&&\left\{\begin{aligned}
\frac{\partial v_1}{\partial \tau}&=\frac{\partial^2 v_1}{\partial \xi^2}+\frac{\partial v_1}{\partial \xi}+Av_2+\dot{s}\left (s\frac{d^2\Theta^0}{d\xi^2}+\frac{\partial v_1}{\partial \xi}\right ), \ \ &\xi<0&,\\
\frac{\partial v_2}{\partial \tau}&=\Le^{-1}\frac{\partial^2 v_2}{\partial \xi^2}+\frac{\partial v_2}{\partial \xi}-Av_2+\dot{s}\left (s\frac{d^2\Phi^0}{d\xi^2}+\frac{\partial v_2}{\partial \xi}\right ), \ \ &\xi<0&,
\end{aligned}\right.\\[1mm]
\label{v2}
&&\left\{\begin{aligned}
\frac{\partial v_1}{\partial \tau}&=\frac{\partial^2 v_1}{\partial \xi^2}+\frac{\partial v_1}{\partial \xi}+\dot{s}\left (s\frac{d^2\Theta^0}{d\xi^2}+\frac{\partial v_1}{\partial \xi}\right ), \ \ &\xi>0&,\\
\frac{\partial v_2}{\partial \tau}&=\Le^{-1}\frac{\partial^2 v_2}{\partial \xi^2}+\frac{\partial v_2}{\partial \xi}+\dot{s}\left (s\frac{d^2\Phi^0}{d\xi^2}+\frac{\partial v_2}{\partial \xi}\right ), \ \ &\xi>0&.
\end{aligned}\right.
\end{eqnarray}
At $\xi=0$, it is easy to see that the new interface conditions are:
\begin{eqnarray*}
[v_1]=[v_2]=0,\qquad\,\,\bigg [\frac{\partial v_1}{\partial \xi}\bigg ]=-s\bigg [\frac{d^2\Theta^0}{d\xi^2}\bigg ],\qquad\,\, \bigg [\frac{\partial v_2}{\partial \xi}\bigg ]=-s\bigg [\frac{d^2\Phi^0}{d\xi^2}\bigg ],\qquad\,\,v_1(\tau,0)=-s\frac{\partial\Theta^0}{\partial\xi}(0).
\end{eqnarray*}
Taking advantage of the conditions
\begin{eqnarray*}
\frac{d\Theta^0}{d\xi}(0)=-\Theta_i,\quad\bigg [\frac{d^2\Theta^0}{d\xi^2}\bigg ]=\frac{\Theta_i}{1-\Theta_i},\quad \bigg [\frac{d^2\Phi^0}{d\xi^2}\bigg ]=-\frac{\Le\Theta_i}{1-\Theta_i},
\end{eqnarray*}
where we used (\ref{eqn:A}) to derive the last condition, it follows that
\begin{equation}\label{transm}
s(\tau)=\frac{v_1(\tau,0)}{\Theta_i},\qquad\;\, \bigg [\frac{\partial v_1}{\partial \xi}\bigg ]=-\frac{v_1(\tau,0)}{1-\Theta_i},\qquad\;\, \bigg [\frac{\partial v_2}{\partial \xi}\bigg ]=\frac{v_1(\tau,0)\Le}{1-\Theta_i}.
\end{equation}

Summarizing,  the free interface problem \eqref{system-1}-\eqref{system-2} has been converted to ($\ref{u_1}$)-($\ref{u_2}$), which constitutes a nonlinear system for $v_1$, $v_2$ and $s$, with transmission conditions \eqref{transm} at $\xi=0$. The next subsections are devoted to the study of the linearized problem (at zero) in an abstract setting, with simplified notation $\pmb u=(u,v)$ for convenience.

\subsection{The linearized problem}\label{linearized subsect}
Now, we consider the linearization at $0$ of the system \eqref{v1}-\eqref{transm}, which reads as follows:
\begin{eqnarray}
\label{linear pb-u}
&&\left\{\begin{aligned}
\frac{\partial u}{\partial \tau}&=\frac{\partial^2u}{\partial\xi^2}+\frac{\partial u}{\partial\xi}+Av, &\xi<0,\\
\frac{\partial v}{\partial\tau}&=\Le^{-1}\frac{\partial^2v}{\partial \xi^2}+\frac{\partial v}{\partial\xi}-Av,\quad &\xi<0,\\
\end{aligned}\right.\\[1mm]
\label{linear pb-v}
&&\left\{\begin{aligned}
\frac{\partial u}{\partial \tau}&=\frac{\partial^2u}{\partial\xi^2}+\frac{\partial u}{\partial\xi}, & \xi>0,\\
\frac{\partial v}{\partial\tau}&=\Le^{-1}\frac{\partial^2v}{\partial\xi^2}+\frac{\partial v}{\partial\xi}, &\quad  \xi>0,
\end{aligned}\right.
\end{eqnarray}
with the interface conditions
\begin{equation}
\label{linear pb-interface}
[u]=[v]=0,\qquad\;\,  \bigg [\frac{\partial u}{\partial\xi}\bigg ]=-\frac{u(\tau,0)}{1-\Theta_i},\qquad\;\,
\bigg [\frac{\partial v}{\partial\xi}\bigg ]=\frac{u(\tau,0)\Le}{1-\Theta_i}.
\end{equation}

Problem \eqref{linear pb-u}-\eqref{linear pb-v} can be written in the more compact form
$\displaystyle\frac{\partial\pmb u}{\partial\tau}={\mathcal L}\pmb u$, where $\pmb u=(u,v)$,
\begin{eqnarray*}
\mathcal{L}=\left(\begin{matrix}
\displaystyle\frac{\partial^2}{\partial\xi^2}+\frac{\partial}{\partial\xi}& &A\chi_{-}\\
0& &\displaystyle\Le^{-1}\frac{\partial^2}{\partial\xi^2}+\frac{\partial}{\partial \xi}-A\chi_{-}
\end{matrix}\right)
\end{eqnarray*}
and $\chi_-$ denotes the characteristic function of the set $(-\infty,0)$.

We now introduce the weighted space $\bm{\mathcal W}$ where we analyze the system \eqref{linear pb-u}-\eqref{linear pb-interface}.
As a matter of fact, the introduction of exponentially weighted spaces for proving stability of traveling waves has been a standard tool since the pioneering work of Sattinger (see \cite{S76}), its role being to shift the continuous spectrum to the left and, thus, creating a gap with the imaginary axis which simplifies the analysis.

\begin{definition}
The exponentially weighted Banach space $\bm{\mathcal W}$ is defined by
\begin{align*}
\bm{\mathcal W}=\Big\{&\pmb u:
e^{\frac{1}{2}\xi}u, e^{\frac{1}{2}\xi}v\in C_b((-\infty,0);\mathbb C),\
e^{\frac{1}{2}\xi}u, e^{\frac{\Le}{2}\xi}v\in C_b((0,\infty);\mathbb C), \lim_{\xi\to 0^{\pm}}u(\xi),\ \lim_{\xi\to 0^{\pm}}v(\xi)\in\mathbb R\Big\},
\end{align*}
equipped with the norm:
\begin{align*}
\|\pmb u\|_{\bm{\mathcal W}}=&\sup_{\xi <0}|e^{\frac{1}{2}\xi}u(\xi)|+\sup_{\xi>0}|e^{\frac{1}{2}\xi}u(\xi)|+\sup_{\xi <0}|e^{\frac{1}{2}\xi}v(\xi)|+\sup_{\xi>0}|e^{\frac{\Le}{2}\xi}v(\xi)|.
\end{align*}	
\end{definition}
In the above definition, $C_b(I;\mathbb C)$ denotes the space of bounded and continuous functions from $I$ to $\mathbb{C}$, $I$ being either the interval $(-\infty,0)$ or $(0,\infty)$. We finally introduce the realization $L$ of the operator ${\mathcal L}$ in $\bm{\mathcal W}$ defined by
\begin{align*}
&D(L)=\bigg\{\pmb u\in \bm{\mathcal W}: \frac{\partial \pmb u}{\partial\xi},\frac{\partial^2\pmb u}{\partial\xi^2}\in \bm{\mathcal W},\ [u]=[v]=0,\ \bigg [\frac{\partial u}{\partial\xi}\bigg ]=-\frac{u(0)}{1-\Theta_i},\ \bigg [\frac{\partial v}{\partial\xi}\bigg ]=\frac{\Le\ u(0)}{1-\Theta_i}\bigg\},\\[1mm]
&L\pmb u=\mathcal{L}\pmb u,\qquad\;\, \pmb u\in \bm{\mathcal W}.
\end{align*}

\begin{remark}
\label{rem-simple}
{\rm We observe that, for any Lewis number, the pair
$\displaystyle\frac{d\pmb U}{d\xi}=\left (\frac{d\Theta^0}{d\xi},\frac{d\Phi^0}{d\xi}\right )$ verifies System \eqref{linear pb-u}, \eqref{linear pb-v}, and it belongs to the space $\bm{\mathcal W}$. In other words, $\displaystyle\frac{d\pmb U}{d\xi}$ is an eigenfunction of the operator $L$ associated with the eigenvalue 0.}
\end{remark}
The above remark gives a first justification for the choice of the exponential weights in the definition of $\bm{\mathcal W}$.
We also stress that, following the same strategy as in the proof of the forthcoming Theorem \ref{thm-2.3} it can be easily checked that the spectrum of the realization of the operator ${\mathcal L}$ in the nonweighted space of pairs $(u,v)$ such that $u$, $v$ are bounded and continuous in $(-\infty,0)\cup (0,\infty)$, contains a parabola which is tangent at $0$ to the imaginary axis.

\subsection{Analysis of the operator $L$}
\label{Resolvent operator}
Next theorem is devoted to a deep study of the operator $L$. For simplicity of notation, for $j=1,2$ we set
\begin{align}
H_{1,\lambda}=\sqrt{1+4\lambda},\qquad\;\,H_{2,\lambda}=\sqrt{\Le^2+4\Le(A+\lambda)},\qquad\;\,H_{3,\lambda}=\sqrt{\Le^2+4\Le\lambda}
\label{formula-1}
\end{align}
and
\begin{align}
&k_{j,\lambda}=\frac{-1+(-1)^{j+1}H_{1,\lambda}}{2},\qquad k_{2+j,\lambda}=\frac{-\Le+(-1)^{j+1}H_{2,\lambda}}{2},\qquad k_{4+j,\lambda}=\frac{-\Le+(-1)^{j+1}H_{3,\lambda}}{2}.
\label{formula-3}
\end{align}

\begin{theorem}
\label{thm-2.3}
The operator $L$ is sectorial and therefore generates an analytic semigroup. Moreover, its spectrum has components:
\begin{enumerate}[\rm (1)]
\item
$(-\infty,-1/4]\cup \mathcal{P}$, where $\mathcal{P}=\{\lambda\in\mathbb{C}:a\Re\lambda+b(\Im\lambda)^2+c\le 0\}$ with
\begin{align*}
a=\bigg (1-\frac{1}{\Le}\bigg )^2,\qquad\;\,
b=\frac{1}{\Le},\qquad\;\,
c=\frac{2A+1}{2}+\frac{8A-5}{4\Le}+\frac{1+A}{\Le^2}-\frac{1}{4\Le^3};
\end{align*}
\item
the simple isolated eigenvalue $0$, the kernel of $L$ being spanned by $\displaystyle\frac{d\pmb U}{d\xi}$;
\item
additional eigenvalues given by the solution of the dispersion relation
\begin{equation}
\label{d,r,Le}
D(\lambda; \Theta_i, \Le):=(k_{6,\lambda}-k_{3,\lambda})(k_{3,\lambda}-k_{2,\lambda})\big [1-(1-\Theta_i)\sqrt{1+4\lambda}\big ]+A\Le,
\end{equation}
where $A$ is given by \eqref{eqn:A}.
\end{enumerate}
\end{theorem}

\begin{proof}
Since the proof is rather lengthy, we split it into four steps.
In the first two steps, we prove properties (1) and (3). Step 3 is devoted to the proof of property (2).
Finally, in Step 4, we prove that the operator $L$ is sectorial in $\bm{\mathcal W}$.

For notational convenience, throughout the proof, we set
\begin{align*}
&{\mathscr I}_1:=\int_{0}^{\infty}f_1(s)e^{-k_1s}ds,&
&{\mathscr I}_2:=\int_{-\infty}^{0}f_1(s)e^{-k_2s}ds,&
&{\mathscr I}_3:=\int_{-\infty}^{0}f_2(s)e^{-k_2s}ds,\\
&{\mathscr I}_4:=\int_{-\infty}^{0}f_2(s)e^{-k_4s}ds, &
&{\mathscr I}_5:=\int_{0}^{\infty}f_2(s)e^{-k_5s}ds,&
\end{align*}
for any fixed $\pmb f=(f_1,f_2)\in\bm{\mathcal W}$, where, here and Step 1 to 3, we simply write $k_j$ instead of $k_{j,\lambda}$ to enlighten the notation.

\vskip 1mm
{\em Step 1}. To begin with, we prove that the interval $(-\infty,-1/4]$ belongs to the point spectrum of $L$. We first assume that $\lambda\le-\Le/4$ (recall that $\Le>1$). In such a case, $\Re(k_1)=\Re(k_2)=-1/2$, $\Re(k_5)=\Re(k_6)=-\Le/2$ and the function $\pmb u$ defined by
\begin{equation}
u(\xi)=
\left\{
\begin{array}{ll}
c_1 e^{k_1\xi}+c_2 e^{k_2 \xi}, & \xi<0,\\
c_5 e^{k_1\xi}+c_6 e^{k_2 \xi}, & \xi\ge 0,
\end{array}
\right.
\qquad\;\,
v(\xi)=
\left\{
\begin{array}{ll}
0, &\xi<0,\\
c_7 e^{k_5\xi}+c_8 e^{k_6\xi}, &\xi\ge 0,
\end{array}
\right.
\label{eigenvalue-1/4}
\end{equation}
belongs to $\bm{\mathcal W}$ and solves the equation $\lambda \pmb u-{\mathcal L}\pmb u={\bf 0}$ for any choice of the complex parameters $c_1$, $c_2$, $c_5$, $c_6$, $c_7$ and $c_8$.
Since there are only four boundary conditions to impose to guarantee that $\pmb u\in D(L)$,
the resolvent equation $\lambda \pmb u-{\mathcal L}\pmb u={\bf 0}$ is not uniquely solvable in $\bm{\mathcal W}$. Thus, $\lambda$ belongs to the point spectrum of $L$.

Next, we consider the case when $\lambda\in(-\Le/4, -1/4]$. In this situation, $\Re(k_1)=\Re(k_2)=-1/2$, however, $\Re(k_5)+\Le/2>0$, $\Re(k_6)+\Le/2<0$.
Thanks to the fact that $e^{\frac{\Le}{2}\xi}v(\xi)$ should be bounded in $(0,\infty)$, the constant $c_7$ in \eqref{eigenvalue-1/4} is zero, whereas the constants $c_1$, $c_2$, $c_5$, $c_6$  $c_8$ are arbitrary. As above, the resolvent equation $\lambda\pmb u-L\pmb u={\bf 0}$ cannot be solved uniquely. Consequently, we conclude that $(-\infty,-1/4]$ belongs to the point spectrum of the operator $L$.

From now on, we consider the case when $\lambda\notin (-\infty,-1/4]$. Then,
$\Re(k_1)+1/2>0$, $\Re(k_2)+1/2<0$, $\Re(k_5)+\Le/2>0$ and $\Re(k_6)+\Le/2<0$.
Similarly to the previous procedure, using the formulae \eqref{rs-u-positive}, \eqref{rs-v-positive} and \eqref{resolvent-u} as well as the fact that the functions $\xi\mapsto e^{\frac{1}{2}\xi}u(\xi)$ and $\xi\mapsto e^{\frac{\Le}{2}\xi}v(\xi)$ should be bounded in $\mathbb{R}$ and in $(0,\infty)$ respectively, the constants $c_2$, $c_5$, $c_7$ can be determined explicitly and they are given by
 \begin{align*}
&c_2=\frac{1}{H_{1,\lambda}}\int_{-\infty}^{0}(Av(s)+f_1(s))e^{-k_2s}ds,\qquad\;\, c_5=\frac{1}{H_{1,\lambda}}{\mathscr I}_1,\qquad\;\,c_7=\frac{\Le}{H_{3,\lambda}}{\mathscr I}_5.
\end{align*}

We now consider formula (\ref{resolvent-v}). Since $\Le>1$, it follows that  $\Re(k_4)+1/2<0$. Moreover, we observe that
the inequality $\Re(k_3)+{1}/{2}\le 0$ is satisfied if and only if $\lambda\in {\mathcal P}$. Indeed, fix any $\lambda\in \stackrel{\circ}{{\mathcal P}}$, the interior of ${\mathcal P}$, so that $\Re(k_3)+1/2<0$, and take
\begin{equation*}
f_1(\xi)=
\left\{
\begin{array}{ll}
e^{-\frac{1}{2}\xi}, &\xi<0,\\
0, &\xi\ge 0,
\end{array}
\right.
\qquad\;\,
f_2\equiv 0 \ \text{in} \ \mathbb{R}.
\end{equation*} 	
In such a case, the more general solution, $\pmb u\in\bm{\mathcal W}$, to the equation
$\lambda \pmb u-{\mathcal L}\pmb u=\pmb f$ is given
by $u(\xi)=c_6e^{k_2\xi}$ and $v(\xi)=c_8e^{k_6\xi}$ for $\xi\ge 0$, whereas $v\equiv 0$ in $(-\infty,0)$ and
$u(\xi)=c_1e^{k_1\xi}+2H_{1,\lambda}^{-2}(2e^{-\frac{1}{2}\xi}-e^{k_1\xi})$
for $\xi<0$. Note that $k_1\neq k_3$ for $\lambda\in\stackrel{\circ}{\mathcal P}$.
Imposing the boundary conditions, we deduce that $c_6=c_8=0$,
$c_1=-2H_{1,\lambda}^{-2}$ and $k_1c_1=2H_{1,\lambda}^{-2}k_2$,
which is clearly a contradiction. We conclude that the domain $\stackrel{\circ}{{\mathcal P}}$ and, consequently, its closure belong to the continuous spectrum of $L$.	
Summarizing, property (1) in the statement of the theorem is established.

\vskip 1mm
{\em Step 2}. Here, we consider the equation $\lambda \pmb u-{\mathcal L}\pmb u=\pmb f$ for $\pmb f\in\bm{\mathcal W}$ and values of $\lambda$ which are not in $(-\infty,-1/4]\cup{\mathcal P}$. For such $\lambda$'s and $j=1,2$ it holds that
\begin{equation}
\Re(k_{2j-1})+\frac{1}{2}>0,\qquad\;\,  \Re(k_{2j})+\frac{1}{2}<0,\qquad\;\, \Re(k_5)+\displaystyle\frac{\Le}{2}>0,\qquad\;\, \Re(k_6)+\displaystyle\frac{\Le}{2}<0.
\label{cond-1}
\end{equation}
We first assume that $k_1\neq k_3$. Imposing that the function $\pmb u$ defined by \eqref{rs-u-positive}-\eqref{resolvent-v} belongs to $\bm{\mathcal W}$, we can
uniquely determine the constants $c_2$, $c_4$, $c_5$ and $c_7$ and we get
\begin{align}
\label{re-u-negative}
u(\xi)
=&c_1e^{k_1\xi}+\frac{e^{k_1\xi}}{H_{1,\lambda}}\int_{\xi}^0 f_1(s)e^{-k_1s}ds
+\frac{e^{k_2\xi}}{H_{1,\lambda}}\int_{-\infty}^{\xi} f_1(s)e^{-k_2s}ds  \nonumber\\
&+\frac{A}{H_{1,\lambda}}\bigg\{\bigg (\frac{e^{k_3\xi}}{k_3-k_2}-\frac{e^{k_3\xi}-e^{k_1\xi}}{k_3-k_1}\bigg )c_3
+\frac{\Le}{H_{2,\lambda}}\bigg[\bigg (\frac{e^{k_1\xi}-e^{k_3\xi}}{k_3-k_1}
-\frac{e^{k_3\xi}}{k_3-k_2}\bigg )\int_{\xi}^0f_2(s)e^{-k_3s}ds\notag\\
&\phantom{-\frac{A}{H_{1,\lambda}}\bigg\{\;\,}+\frac{e^{k_1\xi}}{k_3-k_1}\int_\xi^0 f_2(s)e^{-k_1s}ds
+\bigg (\frac{e^{k_1\xi}-e^{k_4\xi}}{k_4-k_1}+\frac{e^{k_4\xi}}{k_4-k_2}\bigg )\int_{-\infty}^\xi f_2(s)e^{-k_4s}ds\nonumber\\
&\phantom{-\frac{A}{H_{1,\lambda}}\bigg\{\;\,}+\frac{e^{k_1\xi}}{k_4-k_1}\int_\xi^0 f_2(s)(e^{-k_4s}\!-\!e^{-k_1s})ds\!+\! \frac{(k_4-k_3)e^{k_2\xi}}{(k_3-k_2)(k_4-k_2)}\int_{-\infty}^{\xi}f_2(s)e^{-k_2s}ds\bigg]\bigg\},
\\[1mm]
\label{re-v-negative}
v(\xi)&=\bigg (c_3+\frac{\Le}{H_{2,\lambda}}\int_{\xi}^0f_2(s)e^{-k_3s}ds\bigg )e^{k_3\xi}
+\frac{\Le\,e^{k_4\xi}}{H_{2,\lambda}}\int_{-\infty}^{\xi}f_2(s)e^{-k_4s}ds,
\end{align}
for $\xi<0$. Note that $k_2-k_3\neq 0$ (see Appendix \ref{appendix-A}). For $\xi>0$, we get
\begin{align}
\label{re-u-positive}
u(\xi)&=\frac{e^{k_1\xi}}{H_{1,\lambda}}\int_{\xi}^{\infty}f_1(s)e^{-k_1s}ds+\bigg (c_6+{\frac{1}{H_{1,\lambda}}\int_0^{\xi}f_1(s)e^{-k_2s}ds}\bigg )e^{k_2\xi},\\[1mm]
\label{re-v-positive}
v(\xi)&=\frac{\Le\,e^{k_5\xi}}{H_{3,\lambda}}\int_{\xi}^{\infty}f_2(s)e^{-k_5s}ds+\bigg (c_8+\frac{\Le}{H_{3,\lambda}}\int_0^{\xi}f_2(s)e^{-k_6s}ds\bigg )e^{k_6\xi}.
\end{align}

Imposing the boundary conditions, we obtain the following linear system for the unknowns $c_1$, $c_3$, $c_6$ and $c_8$:
\begin{equation}
\begin{pmatrix}
1 &\frac{A}{(k_3-k_2)H_{1,\lambda}} & -1 & 0\\
0 & 1 & 0 & -1\\
k_1 & \frac{Ak_2}{(k_3-k_2)H_{1,\lambda}} & \frac{1}{\Theta_i-1}-k_2 & 0\\
0 & k_3 & \frac{\Le}{1-\Theta_i} & -k_6
\end{pmatrix}
\begin{pmatrix}
c_1\\
c_3\\
c_6\\
c_8
\end{pmatrix}
=\begin{pmatrix}
F_1\\
F_2\\
F_3\\
F_4
\end{pmatrix},
\label{matrix}
\end{equation}
where
\begin{align*}
F_1=&-\frac{A\Le}{(k_4-k_2)H_{1,\lambda}H_{2,\lambda}}{\mathscr I}_4-\frac{1}{H_{1,\lambda}}{\mathscr I}_2+\frac{1}{H_{1,\lambda}}{\mathscr I}_1
-\frac{A\Le(k_4-k_3)}{(k_3-k_2)(k_4-k_2)H_{1,\lambda}H_{2,\lambda}}{\mathscr I}_3;\\[1mm]
F_2=&\frac{\Le}{H_{3,\lambda}}{\mathscr I}_5-\frac{\Le}{H_{2,\lambda}}{\mathscr I}_4;\\[1mm]
F_3=&-\frac{A\Le k_2}{(k_4-k_2)H_{1,\lambda}H_{2,\lambda}}{\mathscr I}_4\!-\!\frac{k_2}{H_{1,\lambda}}{\mathscr I}_2\!+\!\frac{1}{H_{1,\lambda}}\bigg (k_1\!+\!\frac{1}{1-\Theta_i}\bigg ){\mathscr I}_1\!+\!\frac{A\Le k_2}{(k_3-k_2)(k_4-k_2)H_{1,\lambda}}{\mathscr I}_3;\\[1mm]
F_4=&\frac{\Le k_5}{H_{3,\lambda}}{\mathscr I}_5-\frac{\Le k_4}{H_{2,\lambda}}{\mathscr I}_4-\frac{\Le}{(1-\Theta_i)H_{4,\lambda}}{\mathscr I}_1.
\end{align*}
This system is uniquely solvable if and only if $\overline{D}(\lambda;\Theta_i,\Le)=[\Le (k_2-k_3)]^{-1}D(\lambda; \Theta_i, \Le)$, the
determinant of the matrix in left-hand side of \eqref{matrix}, does not vanish, where
$D(\lambda; \Theta_i, \Le)$ is defined in \eqref{d,r,Le}.
Hence, the solutions to the equation $D(\lambda; \Theta_i, \Le)=0$ are elements of the point spectrum of $L$. Property (3) is proved.
On the other hand, as it is easily seen, if $\lambda\notin (-\infty,-1/4]\cup {\mathcal P}$ is not a root of the dispersion relation, then
it is easy to check that the function $\pmb u$ given by \eqref{re-u-negative}-\eqref{matrix} belongs to $D(L)$, so that $\lambda$ is an element of the resolvent set of operator $L$.

Finally, we consider the case when $k_3=k_1$, which gives $\lambda=\lambda_{\pm}:=-\frac{A\Le}{\Le-1}\pm \frac{i\sqrt{A\Le(\Le-1)}}{\Le-1}$ (see Appendices \ref{appendix-A} and \ref{appendix-B}).
It is easy to check that this pair of conjugate complex numbers does not belong to ${\mathcal P}$.
It thus follows that $u$ for $\xi\ge 0$ and $v$ for $\xi\in\R$ are still given by \eqref{re-v-negative}, \eqref{re-u-positive} and \eqref{re-v-positive}.
On the other hand, for $\xi<0$, $u$ is given by
\begin{align*}
u(\xi)=&c_1e^{k_1\xi}-\frac{Ac_3}{H_{1,\lambda}}\xi e^{k_1\xi}+\frac{e^{k_1\xi}}{H_{1,\lambda}}\int_{\xi}^0f_1(s)e^{-k_1s}ds+\frac{e^{k_2\xi}}{H_{1,\lambda}}\int_{-\infty}^{\xi}f_1(s)e^{-k_2s}ds\\
&+\frac{A\Le\, e^{k_1\xi}}{H_{1,\lambda}H_{2,\lambda}}\int_{\xi}^0(s-\xi)f_2(s)ds-\frac{A\Le\, e^{k_1\xi}}{H_{1,\lambda}H_{2,\lambda}^2}\int_{-\infty}^0f_2(s)e^{-k_4s}ds\\
&+\frac{A\Le\ e^{k_1\xi}}{H_{1,\lambda}H_{2,\lambda}^2}\int_{\xi}^0f_2(s)e^{-k_1s}ds+\frac{A\Le\, e^{k_4\xi}}{H_{1,\lambda}H_{2,\lambda}^2}\int_{-\infty}^{\xi}f_2(s)e^{-k_4s}ds\\
&+\frac{A}{H_{1,\lambda}}\bigg\{
\frac{e^{k_1\xi}}{k_1-k_2}c_3+\frac{\Le}{H_{2,\lambda}}\bigg[
\frac{e^{k_4\xi}}{k_4-k_2}\int_{-\infty}^{\xi}f_2(s)e^{-k_4s}ds-\frac{e^{k_1\xi}}{k_1-k_2}\int_{\xi}^0 f_2(s)e^{-k_1s}ds \nonumber\\
&\phantom{+\frac{A}{H_{1,\lambda}}\bigg\{\frac{e^{k_1\xi}}{k_1-k_2}c_3+\frac{\Le}{H_{2,\lambda}}\bigg[\;\,}+ \frac{(k_4-k_1)e^{k_2\xi}}{(k_1-k_2)(k_4-k_2)}\int_{-\infty}^{\xi}f_2(s)e^{-k_2s}ds\bigg]\bigg\}.
\end{align*}
Notice that $\sup_{\xi<0}e^{\frac{1}{2}\xi}|u(\xi)|<\infty$; therefore, $\pmb u$ belongs to $\bm{\mathcal W}$. Imposing the boundary conditions, we get a linear system for the unknowns $(c_1,c_3,c_6,c_8)$, whose matrix is the same as in
\eqref{matrix}. Since the determinant is not zero when $\lambda=\lambda_{\pm}$ (see Appendix \ref{appendix-B}) and the first- and second-order derivatives of
 $\pmb u$ belong to $\pmb {\mathcal W}$, we conclude that $\lambda_{\pm}$ are in the resolvent set of operator $L$.

\vskip 1mm
{\em Step 3}.
Now, we proceed to show that $0$ is an isolated simple eigenvalue of the operator $L$.
In view of the previous steps, in a neighborhood of $\lambda=0$ the solution $\pmb u=R(\lambda,L)\pmb f$ of the equation $\lambda \pmb u-L\pmb u=\pmb f$ is given by
\eqref{re-u-negative}-\eqref{re-v-positive} for any $\pmb f\in\bm{\mathcal W}$, where
\begin{align*}
c_1=&\frac{\Le (k_2\!-\!k_3)}{D(\lambda;\Theta_i,\Le)}\bigg\{
\bigg [\frac{(k_6\!-\!k_3)(1-\Theta_i)}{\Le}\!-\!\frac{A}{(k_3\!-\!k_2)H_{1,\lambda}}\bigg]{\mathscr I}_1\!+\!\frac{k_6\!-\!k_3}{\Le H_{1,\lambda}}{\mathscr I}_2
\!-\!\frac{A(k_6\!-\!k_3)}{(k_3\!-\!k_2)(k_4\!-\!k_2)H_{1,\lambda}}{\mathscr I}_3\\
&\phantom{\frac{\Le (k_2\!-\!k_3)}{D(\lambda;\Theta_i,\Le)}\bigg\{\,}
+\frac{A}{H_{1,\lambda}H_{2,\lambda}}\bigg(\frac{k_6-k_3}{k_4-k_2} -\frac{k_6-k_4}{k_3-k_2} \bigg){\mathscr I}_4-\frac{A}{(k_3-k_2)H_{1,\lambda}}{\mathscr I}_5\bigg\},\\[1mm]
c_3=&\frac{\Le (k_2\!-\!k_3)}{D(\lambda;\Theta_i,\Le)}\bigg\{{\mathscr I}_1+{\mathscr I}_2-\frac{A\Le}{(k_4-k_2)(k_3-k_2)}{\mathscr I}_3\\
&\phantom{\frac{\Le (k_2\!-\!k_3)}{D(\lambda;\Theta_i,\Le)}\bigg\{\;\,}
+\frac{1}{H_{2,\lambda}}\bigg [(k_6\!-\!k_4)\big[1-H_{1,\lambda}(1-\Theta_i)\big]\!+\!\frac{A\Le }{k_4-k_2} \bigg]{\mathscr I}_4
\!+\!\big[1-H_{1,\lambda}(1-\Theta_i)\big]{\mathscr I}_5\bigg\},\\[1mm]
c_6=&\frac{\Le (k_2\!-\!k_3)}{D(\lambda;\Theta_i,\Le)}\bigg\{\frac{1}{H_{1,\lambda}}\bigg (\frac{A}{k_3-k_2}+\frac{k_6-k_3}{\Le}\bigg ){\mathscr I}_1
\!+\!\frac{(k_6\!-\!k_3)(1\!-\!\Theta_i)}{\Le}{\mathscr I}_2\!-\!\frac{A(k_6\!-\!k_3)(1\!-\!\Theta_i)}{(k_3\!-\!k_2)(k_4\!-\!k_2)}{\mathscr I}_3
\\
&\phantom{\frac{\Le (k_2-k_3)}{D(\lambda;\Theta_i,\Le)}\bigg\{\;\,}
\!+\!\frac{A(1\!-\!\Theta_i)}{H_{2,\lambda}}\bigg(\frac{k_6\!-\!k_3}{k_4\!-\!k_2}\!-\!\frac{k_6\!-\!k_4}{k_3\!-\!k_2}\bigg){\mathscr I}_4
-\frac{A(1\!-\!\Theta_i)}{k_3\!-\!k_2}{\mathscr I}_5
\bigg\},\\[1mm]
c_8=&\frac{\Le (k_2\!-\!k_3)}{D(\lambda;\Theta_i,\Le)}\bigg\{{\mathscr I}_1\!+\!{\mathscr I}_2\!-\!\frac{A\Le}{(k_3\!-\!k_2)(k_4\!-\!k_2)}{\mathscr I}_3+\bigg[1\!-\!H_{1,\lambda}(1\!-\!\Theta_i)\!+\!\frac{A\Le}{(k_3\!-\!k_2)(k_4\!-\!k_2)}\bigg]{\mathscr I}_4
\\
&\phantom{\frac{\Le (k_2\!-\!k_3)}{D(\lambda;\Theta_i,\Le)}\bigg\{}
+\bigg[\frac{A\Le}{(k_3-k_2)H_{3,\lambda}}
+[1-H_{1,\lambda}(1-\Theta_i)]\bigg (1+\frac{k_6-k_3}{H_{3,\lambda}}\bigg )\bigg]{\mathscr I}_5\bigg\}.
\end{align*}
As it is immediately seen, the function $D(\cdot;\Theta_i,\Le)$ is analytic in a neighborhood of $\lambda=0$, which is simple zero of such a function, and the other functions appearing in \eqref{re-u-negative}-\eqref{re-v-positive} are holomorphic in a neighborhood of $\lambda=0$. Hence, we conclude that zero is a simple pole of the resolvent operator $R(\lambda,L)$.
Since $\displaystyle\frac{d\pmb U}{d\xi}$ belongs to the kernel of $L$ (see Remark \ref{rem-simple}) and the matrix in \eqref{matrix} has rank three at $\lambda=0$, this function
generates the kernel, so that the geometric multiplicity of the eigenvalue $\lambda=0$ is one. This is enough to conclude that
$\lambda=0$ is a simple eigenvalue of $L$. Property (2) is established and the spectrum of $L$ is completely characterized.

\vskip 1mm
{\em Step 4}. In order to prove that $L$ is sectorial, it is sufficient to show that there exist two positive constants $C$ and $M$ such that
\begin{align}
\label{resolvent estimate}
\|R(\lambda,L)\|_{L(\bm{\mathcal W})}\le
C|\lambda|^{-1},\qquad\;\,\Re{\lambda}\ge M.
\end{align}
Without loss of generality, we can assume that $k_{1,\lambda}\neq k_{3,\lambda}$ and the conditions in \eqref{cond-1} are all satisfied if $\Re\lambda\ge M$.
Throughout this step, $C_j$ denotes a positive constant, independent of $\lambda$ and $\pmb f\in\bm{\mathcal W}$.

We begin by estimating the terms $H_{j,\lambda}$ ($j=1,2,3$). As it is easily seen,
\begin{align}
|H_{2,\lambda}|\ge\Re(H_{2,\lambda})=\sqrt{\frac{|{\Le}^2+4\Le(A+\lambda)|+{\Le}^2+4\Le(A+\Re\lambda)}{2}}\ge\sqrt{2\Le|\lambda|}
\label{estim-H1}
\end{align}
for any $\lambda\in\C$ with positive real part.
Since $H_{1,\lambda}$ and $H_{3,\lambda}$ can be obtained from $H_{2,\lambda}$, by taking, $(\Le,A)=(1,0)$ and $(\Le,A)=(\Le,0)$ respectively, we also deduce that
\begin{align}
|H_{1,\lambda}|\ge\Re(H_{1,\lambda})\ge\sqrt{2|\lambda|},\qquad\;\,|H_{3,\lambda}|\ge\Re(H_{3,\lambda})\ge\sqrt{2\Le|\lambda|}
\label{estim-H2-H3}
\end{align}
for the same values of $\lambda$.
Thanks to \eqref{estim-H1} and \eqref{estim-H2-H3}, we can easily estimate the terms
${\mathscr I}_j$ $(j=1,\ldots,5)$. Indeed, since $\Re(k_1)+1/2>0$, we obtain
\begin{align*}
|{\mathscr I}_1|&=\bigg |\int_0^{\infty}f_1(s)e^{-k_1s}ds\bigg |\le \sup_{\xi>0}e^{\frac{1}{2}\xi}|f_1(\xi)|\int_0^{\infty}e^{-\frac{1}{2}\Re(H_{1,\lambda})s}ds \le C_1|\lambda|^{-\frac{1}{2}}\|\pmb f\|_{\bm{\mathcal W}}.
\end{align*}
The other terms ${\mathscr I}_j$ can be treated likewise and we get
$\sum_{j=2}^5|{\mathscr I}_j|\le C_2|\lambda|^{-\frac{1}{2}}\|\pmb f\|_{\bm{\mathcal W}}$ for every $\pmb f\in\bm{\mathcal W}$ and $\lambda\in\C$ with positive real part.

Next, we turn to the function $D(\cdot;\Theta_i,\Le)$. We observe that
\begin{align*}
|D(\lambda;\Theta_i,\Le)|\ge [(1-\Theta_i)\sqrt{|1+4\lambda|}-1]|k_{6,\lambda}-k_{3,\lambda}||k_{3,\lambda}-k_{2,\lambda}|-A\Le
\end{align*}
for any $\lambda\in\C$.
Taking \eqref{estim-H1} and \eqref{estim-H2-H3} into account, we can show that
\begin{equation}
C_3\sqrt{|\lambda|}\le |k_{3,\lambda}-k_{2,\lambda}|+|k_{3,\lambda}-k_{6,\lambda}|\le C_4\sqrt{|\lambda|}
\label{estimate-k1-k4}
\end{equation}
for $\lambda\in\C$ with sufficiently large positive real part. Hence, for such values of $\lambda$'s we can continue the previous inequality and get
\begin{align}
|D(\lambda;\Theta_i,\Le)|\ge C_5|\lambda|^{\frac{3}{2}}.
\label{estimate-D}
\end{align}
Similarly, $|k_{6,\lambda}-k_{4,\lambda}|\le C_6\sqrt{|\lambda|}$
for any $\lambda$ with positive real part and
\begin{equation}
|k_{4,\lambda}-k_{2,\lambda}|\ge \frac{1}{2}|H_{2,\lambda}|-\frac{1}{2}|H_{1,\lambda}|-\frac{\Le-1}{2}
\ge \sqrt{\frac{\Le|\lambda|}{2}}-\sqrt{\frac{|\lambda|}{2}}-\frac{\Le-1}{2}\ge C_7\sqrt{|\lambda|},
\label{estim-k2-k4}
\end{equation}
if $\Re\lambda$ is sufficiently large. From \eqref{estim-H1}-\eqref{estim-k2-k4} we infer that
$|c_1|+|c_3|+|c_6|+|c_8|\le C_8|\lambda|^{-1}$ for any $\lambda\in\C$ with $\Re(\lambda)\ge M$ and
a suitable positive constant $M$. Further, observing that
\begin{eqnarray*}
|k_{3,\lambda}-k_{1,\lambda}|+|k_{4,\lambda}-k_{1,\lambda}|\ge C_9\sqrt{|\lambda|},\qquad\;\,|k_{4,\lambda}-k_{3,\lambda}|\le C_{10}\sqrt{|\lambda|},
\end{eqnarray*}
we are now able to estimate the functions $u$ and $v$ in \eqref{re-u-negative}-\eqref{re-v-positive} and show that
\eqref{resolvent estimate} holds true. The proof is complete.
\end{proof}

\begin{remark}
\rm{It is worth pointing out that, as $\Le \to \infty$, the set ${\mathcal P}$ degenerates into a vertical line $\Re \lambda=-\Theta_i(1-\Theta_i)^{-1}-1/2$. In the limit case, the system is partly parabolic and the semigroup is not analytic, see, e.g., \cite[Section 1, p. 2435]{GLS10}.}
\end{remark}

\setcounter{tocdepth}{2}
\section{The fully nonlinear problem}
\label{sect-3}
Our goal in this section is to get rid of the eigenvalue 0 and then derive a new fully nonlinear problem.
We recall that the eigenvalue $0$ is related to the translation invariance of the traveling wave. In a first step, we use here a method similar to that of \cite{BLS92} or \cite[p. 358]{Lunardi96}.

\subsection{Ansatz revisited: elimination of the eigenvalue $0$}
\label{subsect-2.3}
It is convenient to write System \eqref{u_1}-\eqref{u_2} with notation $\pmb u=(u_1,u_2)$, $\pmb U=(\Theta^0,\Phi^0)$, see Section \ref{fixed}, in an abstract form:
\begin{equation}
\label{v}
\dot{\pmb u}=L\pmb u+\dot{s}{\pmb U}^{\prime}+\dot{s}{\pmb u}^{\prime}.
\end{equation}
Note that, in view of \eqref{interface-u}, $\pmb u(\tau,\cdot)$ belongs to $D(L)$ for each $\tau$.
Since 0 is an isolated simple eigenvalue of $L$, we can introduce the spectral projection $P$ onto the kernel of $L$, defined by
$P\pmb f=\langle\pmb f,{\pmb e^{*}}\rangle\pmb U^{\prime}$ for every $\pmb f\in\bm{\mathcal W}$ and a unique ${\pmb e^{*}}\in \bm{\mathcal W}^*$, the dual space of $\bm{\mathcal W}$, such that $\langle\pmb U^{\prime},{\pmb e^{*}}\rangle=1$. For further use, we recall that $P$ commutes with $L$ on $D(L)$.
We are going to apply the projections $P$ and $Q=I-P$ to System ($\ref{v}$) to remove the eigenvalue $0$.

\medskip
\paragraph{\bf Ansatz 2} We split $\pmb u$ into $\pmb u(\tau,\cdot)=P\pmb u(\tau,\cdot)+Q\pmb u(\tau,\cdot)=p(\tau)\pmb U^{\prime}+\pmb w(\tau,\cdot)$, i.e.,
\begin{align}
u_1(\tau, \xi)=&p(\tau)\frac{d\Theta^0}{d\xi}(\xi)+w_1(\tau, \xi),\label{splitting-w1}\\
 u_2(\tau, \xi)=&p(\tau)\frac{d\Phi^0}{d\xi}(\xi)+w_2(\tau, \xi),\notag
\end{align}
where $p(\tau)=\langle \pmb u(\tau),{\pmb e^{*}}\rangle $ and $\pmb w=(w_1,w_2)$. Clearly, $\pmb w(\tau,\cdot)\in Q(D(L))$ for each $\tau$.
It follows from ($\ref{v}$)  that
\begin{equation}
\dot{p}=\dot{s}+\dot{s}\langle\pmb u^{\prime},{\pmb e^{*}}\rangle,\qquad\;\,
\dot{\pmb w}=L\pmb w+\dot{s}Q\pmb u^{\prime},
\label{w}
\end{equation}
a Lyapunov-Schmidt-like reduction of the problem. We point out that the above procedure generates a new ansatz slightly different from ansatz 1 (see \eqref{ansatz1}) that helps us determine the functional framework.

Thanks to new ansatz 2, we are going to derive an equation for $\pmb w$ in the space $\bm{\mathcal W}$. Now, the spectrum of the part of $L$ in $Q(\bm{\mathcal W})$ does not contain the eigenvalue $0$.

\subsection{Derivation of the fully nonlinear equation}
To get a self-contained equation for $\pmb w$, we need to eliminate $\dot{s}$ from the right-hand side of the second equation in \eqref{w}. For this purpose,
we begin by evaluating the first component of \eqref{w} at $\xi=0^+$ to get
\begin{align}
\frac{\partial w_1}{\partial\tau}(\cdot,0^+)=&(L\pmb w)_1(\cdot,0^+)+\dot{s}(Q\pmb u')_1(\cdot,0^+)\notag\\
=&(L\pmb w)_1(\cdot,0^+)+\dot{s}\frac{\partial u_1}{\partial \xi}(\cdot,0^+)+\dot{s}\langle\pmb u',{\pmb e^{*}}\rangle\Theta_i.
\label{e}
\end{align}
Next, we observe that the function $w_1$ is continuous (but not differentiable) at $\xi=0$, since both $\pmb u$ and $\pmb U'$ are continuous at $\xi=0$. Therefore, evaluating \eqref{splitting-w1} at $\xi=0$ and recalling that $u_1(\tau,0)=0$ (see \eqref{interface-u}), we infer that
$w_1(\tau,0)=\Theta_ip(\tau)$. Differentiating this formula yields
\begin{align}
\frac{\partial w_1}{\partial\tau}(\cdot,0)=\dot{p}\Theta_i=\dot{s}\Theta_i+\dot{s}\langle \pmb u',{\pmb e^{*}}\rangle\Theta_i,
\label{key}
\end{align}
From  ($\ref{e}$) and ($\ref{key}$), it follows that
\begin{equation}
\label{s}
\dot{s}\Theta_i=(L\pmb w)_1(\cdot,0^+)+\dot{s}\frac{\partial u_1}{\partial \xi}(\cdot,0^+).
\end{equation}
To get rid of the spatial derivatives of $u_1$ from the right-hand side of \eqref{s}, we use \eqref{splitting-w1} to write
\begin{eqnarray}
\label{u_1,w_1}
\frac{\partial u_1}{\partial \xi}(\cdot,0^+)=p\frac{d^2\Theta^0}{d\xi^2}(0^+)+w'_1(\cdot,0^+)
=w_1(\cdot,0)+w'_1(\cdot,0^+).
\end{eqnarray}
Plugging ($\ref{u_1,w_1}$) into ($\ref{s}$), we finally obtain the formula
\begin{equation}\label{shift}
\dot s=\frac{(L\pmb w)_1(\cdot,0^+)}{\Theta_i-w_1(\cdot,0)-w'_1(\cdot,0^+)},
\end{equation}
which can be regarded as a underlying \textit{second-order Stefan condition}, see \cite{BL18}.
Hence, replacing it in ($\ref{w}$), we get
\begin{align*}
\frac{\partial\pmb w}{\partial\tau}
=&L\pmb w+\frac{(L\pmb w)_1(\cdot,0^+)}{\Theta_i-w_1(\cdot,0)-w'_1(\cdot,0^+)}Q\pmb u' \nonumber\\
=&L\pmb w+\frac{(L\pmb w)_1(\cdot,0^+)}{\Theta_i-w_1(\cdot,0)-w'_1(\cdot, 0^+)}Q\bigg (\frac{w_1(\cdot,0)}{\Theta_i}\pmb{U''}+\pmb{w'}\bigg ),
\end{align*}
which is a fully nonlinear parabolic equation in the space $\bm{\mathcal W}$ written in a more abstract form:
\begin{equation}\label{FNLE}
\frac{\partial\pmb w}{\partial\tau} = L\pmb w + {F}(\pmb w), \quad  {\pmb w}\in Q(D(L)).
\end{equation}
and is going to be the subject of our attention.
Note that Equation \eqref{FNLE} is fully nonlinear since the function $F$ depends on $\pmb w$ also through the limit at $0^+$ of $L\pmb w$.
Moreover, the operator $L$ is sectorial in $Q(\bm{\mathcal W})$.
Hence, we can take advantage of the theory of analytic semigroups to solve Equation \eqref{FNLE}. We refer the reader to \cite[Chapter 4]{Lunardi96} for further details.

\setcounter{tocdepth}{2}
\section{Stability of the traveling wave solution}\label{stability}
This section is devoted to the analysis of the stability of the traveling wave solution  $\pmb U$. Here, stability refers to orbital stability with asymptotic phase $s_{\infty}$.	
 From now on, we focus on the asymptotic situation where the Lewis number, $\Le$, is large and, in this respect, we use the notation  $\varepsilon = 1/\Le$ to stand for a small perturbation parameter. Simultaneously, we assume that $\Theta_i$ is close to the burning temperature normalized at unity, which is physically relevant (see \cite[Section 3.2, Fig. 5]{BGKS15}). More specifically, we restrict $\Theta_i$ to the domain  $\frac{2}{3}<\Theta_i <1$.

In what follows, we introduce \hbox{$m:= \Theta_i/(1-\Theta_i)$} as the \textit{bifurcation parameter} which runs in the interval $(2,\infty)$, due to the choice of $\Theta_i$.
With the above notation, $A =m+\varep m^2$ and the \textit{dispersion relation} $D(\lambda;\Theta_i,\Le)$ (see \eqref{d,r,Le}) in Section 2 reads:
\begin{align}
\label{dispersion epsilon}
D_{\varepsilon}(\lambda;m) =&-\frac{1}{4}\big(\sqrt{1+4\varepsilon(m+\varepsilon m^2+\lambda)}+\sqrt{1+4\varepsilon \lambda}\big)\notag\\
&\qquad\times\bigg(\frac{1}{\varepsilon}[\sqrt{1+4\varepsilon(m+\varepsilon m^2+\lambda)}-1]\!+\!1\!+\!\sqrt{1+4\lambda}\bigg)\!
\bigg(1\!-\!\frac{\sqrt{1+4\lambda}}{1+m}\bigg )\!+\!m\!+\!\varepsilon m^2.
\end{align}

This section is split into two parts. First, we study the stability of the null solution of the fully nonlinear equation \eqref{FNLE}. Second, we turn our attention to the stability of the traveling wave.

\subsection{Stability of the null solution of (\ref{FNLE})}
To begin with, we recall that the spectrum of the part of $L$ in $\bm{\mathcal W}_Q:=Q(\bm{\mathcal W})$ is the set
\begin{align*}
\left (-\infty,-{\textstyle \frac{1}{4}}\right ]\cup\mathcal{P}\cup\{\lambda\in\C\setminus\{0\}:D_{\varepsilon}(\lambda;m)=0\}.
\end{align*}
As we will show, the roots of the dispersion relation $D_{\varepsilon}(\cdot;m)$ are finitely many.
As a consequence, there is a gap between the spectrum of this operator and the imaginary axis (at least for $\varepsilon$ small enough).
In view of the principle of linearized stability, the main step in the analysis of the stability of the null solution of Equation \eqref{FNLE} is a deep insight in the solutions of the dispersion relation. More precisely, we need to determine when they are all contained in the left halfplane and when some of them lie in the right halfplane.

The limit critical value $m^c=6$ will play an important role in the analysis hereafter.

\begin{theorem}
\label{stability theorem FNLE}
The following properties are satisfied.
\begin{enumerate}[\rm (i)]
\item
Let $m\in (2,m^c)$ be fixed. Then, there exists $\varepsilon_0=\varepsilon_0(m)>0$ such that, for $\varepsilon\in (0,\varepsilon_0)$, the null solution of the fully nonlinear problem \eqref{FNLE} is stable with respect to perturbations belonging to $Q(D(L))$.
\item
Let $m>m^c$ be fixed. Then, there exists $\varepsilon_1=\varepsilon_1(m)$ small enough such that, for $\varepsilon\in (0,\varepsilon_1)$, the null solution of \eqref{FNLE} is unstable with respect to perturbations belonging to $Q(D(L))$.
\end{enumerate}
\end{theorem}

\begin{proof}
To begin with, we observe that the functions $D_{\varepsilon}(\cdot,m)$ are holomorphic in $\C\setminus (-\infty,-1/4]$ and therein they locally converge to the \textit{limit dispersion relation} $D_0(\cdot,m)$ defined by
\begin{align*}
D_0(\lambda;m)=&-\frac{1}{2}[2(m+\lambda)+1+\sqrt{1+4\lambda}]\left (1-\frac{\sqrt{1+4\lambda}}{1+m}\right )+m\notag\\
=&\frac{\sqrt{1+4\lambda}-1}{4(1+m)}[4\lambda-(m-2)\sqrt{1+4\lambda}+m+2],
\end{align*}	
as $\varepsilon\to 0^+$.
The solutions of the equation $D_0(\lambda;m)=0$ are $\lambda=0$, for all $m$,
and the roots of the second-order polynomial $4\lambda^2+(6m-m^2)\lambda+2m$, whose real part is not less than $-(m+2)/4$.
This polynomial admits conjugate solutions $\lambda_{1,2}= a(m) \pm ib(m)$, where $a(m)=\frac{1}{8}(m^2-6m)$ and $b(m)= \frac{1}{8}(m-2)\sqrt{|8m-m^2|}$, if $m\in (2,8)$ and real solutions $\lambda_{1,2}= a(m) \pm b(m)$ otherwise. The coefficient $a(m)$ is negative whenever $2<m<6$ and positive for $m>6$. It can be easily checked
that ${\rm Re}(\lambda_{1,2})\ge -(m+2)/4$ for each $m\in (2,\infty)$, so that $\lambda_{1,2}$ solve the equation $D_0(\lambda; m)=0$. In particular, there are two conjugate purely imaginary roots $\lambda_{1,2}=\pm\sqrt{3}i$ at $m=6$.

We can now prove properties (i) and (ii).

(i) Fix $\rho>0$ such that the closure of the disks of center $\lambda_{1,2}$ and radius $\rho$ is contained in $\{\Re z <0\}\backslash (-\infty,-\frac{1}{4}]$. Hurwitz Theorem (see, e.g., \cite[Chapter 7, Section 2]{Conway78}) and the above results show that there exists $\varepsilon_0 >0$ such that, for $\varepsilon\in (0,\varepsilon_0)$, $D_{\varepsilon}(\lambda;m)$ admits exactly two conjugate complex roots $\lambda_{1,2}(\varepsilon)$ in the disk $|\lambda-\lambda_{i}|<\rho$ and $\lambda_{i}(\varepsilon)$ converges to $\lambda_i$, as $\varepsilon \to 0$, for $i=1,2$. Therefore, all the elements of the spectrum of the part of operator $L$
in $\bm{\mathcal W}_Q$ have negative real parts, which implies that the operator norm of the restriction to $\bm{\mathcal W}_Q$ of the analytic semigroup $e^{\tau L}$ generated by $L$, decays to zero with exponential rate as $t\to\infty$. Now, the nonlinear stability follows from applying a standard machinery: the solution of Equation \eqref{FNLE}, with initial datum $\pmb w_0$ in a small (enough) ball of $Q(D(L))$ centered at zero, is given by the variation-of-constants-formula
\begin{eqnarray*}
\pmb w(\tau,\cdot)=e^{\tau L}\pmb w_0+\int_0^{\tau}e^{(\tau-s)L}F(\pmb w(s,\cdot))ds,\qquad\;\,\tau>0.
\end{eqnarray*}
Applying the Banach fixed point theorem in the space
\begin{eqnarray*}
\bm{\mathcal X}^{\alpha}_{\omega}\!=\!\bigg\{\pmb w\!\in\! C([0,\infty);\pmb{\mathcal W}_Q):\sup_{\sigma\in (0,1)}\sigma^{\alpha}\|\pmb w\|_{C^{\alpha}([\sigma,1];D(L))}<\infty: \tau\mapsto e^{\omega\tau}\pmb w(\tau,\cdot)\!\in\! C^{\alpha}([1,\infty);D(L))\bigg\},
\end{eqnarray*}
endowed with the natural norm, where $\alpha$ is fixed in $(0,1)$ and $\omega$ is any positive number less than the real part of
$\lambda_1(\varepsilon)$, allows us to prove the existence and uniqueness of a solution $\pmb w$ of \eqref{FNLE}, defined in $(0,\infty)$ such that
$\|\pmb w(\tau,\cdot)\|_{\pmb{\mathcal W}}+\|L{\pmb w}(\tau,\cdot)\|_{\pmb{\mathcal W}}
\le Ce^{-\omega\tau}\|\pmb w_0\|_{D(L)}$ for $\tau\in (0,\infty)$ and some positive constant $C$, which yields the claim. For further details see \cite[Chapter 9]{Lunardi96}.

(ii) For $m>m^c$, we use again Hurwitz Theorem to show that there exists $\varepsilon_1=\varepsilon_1(m)>0$ such that the equation
$D_{\varepsilon}(\lambda,m)=0$ admits a solution with positive real part if $\varepsilon\in (0,\varepsilon_1)$. More precisely, it admits a couple of conjugate complex roots with positive real parts, if $m<8$, a positive root, if $m=8$, and two real solutions if $m>8$. For these values of $\varepsilon$, the restriction of the semigroup $e^{\tau L}$ to $\bm{\mathcal W}_Q$  exhibits an exponential dichotomy, i.e., there exists a spectral projection $P_+$ which allows to split $\bm{\mathcal W}_Q=P_+(\bm{\mathcal W}_Q)\oplus (I-P_+)(\bm{\mathcal W}_Q)$. The semigroup $e^{\tau L}$ decays to zero with exponential rate when restricted to $(I-P)(\bm{\mathcal W}_Q)$, whereas the restriction of $e^{\tau L}$ to $P_+(\bm{\mathcal W}_Q)$
extends to a group which decays to zero with exponential rate as $\tau\to-\infty$. Again with a fixed point technique, we can prove the existence of a nontrivial backward solution
$\pmb z$ of the nonlinear equation \eqref{FNLE}, defined in $(-\infty,0)$ such that
$\|\pmb z(\tau,\cdot)\|_{\pmb{\mathcal W}}+\|L\pmb z(\tau,\cdot)\|_{\pmb{\mathcal W}}\le C_{\omega}e^{\omega\tau}$ for $\tau\in (-\infty,0)$ and
any $\omega$ positive and smaller than the minimum of the positive real parts of the roots of the dispersion relation.
The sequence $(\pmb z_n)$ defined by $\pmb z_n=\pmb z(-n,\cdot)$ vanishes in $D(L)$ as $n\to+\infty$ and the solution $\pmb w_n$ to \eqref{FNLE} subject to the initial condition
$\pmb w_n(0,\cdot)=\pmb z_n$ exists at least in the time domain $[0,n]$, where it coincides with the function $\pmb z(\cdot-n,\cdot)$. Thus,
the norm of $\|\pmb w_n\|_{C([0,n];\pmb{\mathcal W}_Q)}$ is positive and far way from zero, uniformly with respect to $n\in\N$, whence the instability of the trivial solution of
\eqref{FNLE} follows. Again, we refer the reader to \cite[Chapter 9]{Lunardi96} for further results.
\end{proof}
     	
\subsection{Stability of the traveling wave}
We can now rewrite the results in Theorem \ref{stability theorem FNLE} in terms of problem
\eqref{perturbation T}-\eqref{interface-Theta-Phi}.
\begin{theorem}
\label{stability theorem TW}
The following properties are satisfied.
\begin{enumerate}[\rm (i)]
\item
For $m\in (2,m^c)$ fixed, there exists $\varepsilon_0=\varepsilon_0(m)>0$ such that, for $\varepsilon\in (0,\varepsilon_0)$, the traveling wave solution $\pmb U$ is orbitally stable with asymptotic phase $s_{\infty}$ $($see \eqref{s-infty}$)$, with respect to perturbations belonging to the weighted space $D(L)$.
\item
For $m>m^c$ fixed, there exists $\varepsilon_1=\varepsilon_1(m)$ small enough such that, for $\varepsilon\in (0,\varepsilon_1)$, the traveling wave $\pmb U$ is unstable.
with respect to perturbations belonging to the weighted space $D(L)$.
\end{enumerate}
\end{theorem}

\begin{proof}
(i) Let us fix $\pmb w_0\in Q(D(L))$ with $\|\pmb w_0\|_{D(L)}$ small enough, so that Theorem \ref{stability theorem FNLE}(i) can be applied.
Denote by $\pmb w$ the classical solution to Equation \eqref{FNLE} which satisfies the initial condition $\pmb w(0,\cdot)=\pmb w_0=(w_{0,1},w_{0,2})$.
Observe that, since $p=\Theta_i^{-1}w_1(\cdot,0)$ (see Subsection \ref{subsect-2.3}) it follows that the problem \eqref{v}, subject to the initial condition $\pmb u(0,\cdot)=\Theta_i^{-1}w_{0,1}\pmb U'+\pmb w_0$, admits a unique classical solution $(\pmb u,s)$, where $\pmb u$ decreases to zero as $\tau\to\infty$, with exponential rate. Moreover, using \eqref{shift} it is immediate to check
that $s(\tau)$ converges to
\begin{equation}
s_\infty=\int_{0}^{\infty}\frac{(L\pmb w)_1(\tau,0^+)}{\Theta_i-w_1(\tau,0)-w'_1(\tau,0^+)}d\tau,
\label{s-infty}
\end{equation}
as $\tau\to\infty$ (assuming for simplicity that $g$ vanishes at $\tau=0$).
We point out that $s_{\infty}$ depends on the initial condition.

Coming back to problem \eqref{perturbation T}-\eqref{interface-Theta-Phi} with initial condition ${\pmb X}(0)=\pmb u_0+\pmb U$ and $g(0)=0$, we easily see that the solution ${\pmb X}=(\Theta,\Phi)$ is defined by
\begin{align*}
&{\pmb X}=p\pmb U'+\pmb w+\pmb U=\Theta_i^{-1}w_1(\cdot,0)\pmb U'+\pmb w+\pmb U,\\
&g(\tau)=\tau+\int_0^{\tau}\frac{(L\pmb w)_1(\sigma,0^+)}{\Theta_i-w_1(\sigma,0)-w'_1(\sigma,0^+)}d\sigma,\qquad\;\,\tau\ge 0.
\end{align*}
From this formula and the above result, the claim follows at once.

(ii) The proof is similar to that of property (i) and, hence, it is left to the reader.
\end{proof}

\setcounter{tocdepth}{2}

\section{Hopf bifurcation}
\label{sect-5}
This section is devoted to investigating the dynamics of the perturbation of the traveling wave in a neighborhood, say $(6-\delta, 6+\delta)$, of the limit critical value $m^c=6$ (see Section \ref{stability}).  As regards parameter $m$, the situation is more complicated than in Section 4 when it was fixed. Now, the dispersion relation ${D}_{\varepsilon}(\lambda;m)$ can be seen as a sequence of analytic functions parameterized by $m$. The main difficulty here is that Hurwitz Theorem does not a priori apply, particularly because of the lack of uniformity of ${D}_{\varepsilon}(\lambda;m)$ with respect to $\varep$ and $m$. We especially find a proper approach to combining $m$ with  $\varep$:  we construct in Proposition \ref{give critical value}  a sequence of critical values $m^c(\varep)$ such that $m^c(0)=m^c$ and apply Hurwitz Theorem to the sequence $D_{\varep}(\lambda,m^c(\varep))$. This proposition will be crucial for proving the existence of a Hopf bifurcation (see Theorem \ref{Hopf bifurcation theorem}).

\subsection{Local analysis of the dispersion relation}\label{p7}
We look for the roots of the \textit{dispersion relation}, see  \eqref{dispersion epsilon},
in a neighborhood of $m^c=6$ and of $\lambda = \pm i\sqrt{3}$, for $\varep>0$ small enough. A natural idea is to turn the dispersion relation into a polynomial by squaring, however the price to pay is double: the polynomial will be of high order without algebraic solution, and
spurious roots therefore appear.

For convenience, we rewrite the equation $D_{\varepsilon}(\lambda;m)=0$ into a much more useful form. Replacing
$\sqrt{1+4\varepsilon(m+\varepsilon m^2+\lambda)}+\sqrt{1+4\varepsilon \lambda}$ by
$4\varepsilon(m+\varepsilon m^2)(\sqrt{1+4\varepsilon(m+\varepsilon m^2+\lambda)}-\sqrt{1+4\varepsilon \lambda})^{-1}$ with some straightforward algebra
we obtain the equivalent equation
\begin{equation}
\sqrt{1+4\varepsilon \lambda}-\frac{1}{1+m}\sqrt{1+4\varepsilon(m+\varepsilon m^2+\lambda)}\sqrt{1+4\lambda}+\frac{1+\varepsilon m}{1+m}\sqrt{1+4\lambda}
=\varepsilon\frac{1+4\lambda}{1+m}+1-\varepsilon.
\label{zeta}
\end{equation}

If we denote by $\zeta$ the right-hand side of \eqref{zeta} and set
\begin{align*}
\Sigma_1=&1+4\varepsilon\lambda+\frac{2+6\varepsilon m+5\varepsilon^2 m^2+4\varepsilon\lambda}{(1+m)^2}(1+4\lambda),\\[1mm]
\Sigma_2=&\frac{1+4\lambda}{(1+m)^2}\bigg [(2+6\varepsilon m+5\varepsilon^2 m^2+4\varepsilon\lambda)(1+4\varepsilon\lambda)
+\frac{[1+4\varepsilon(m+\varepsilon m^2+\lambda)](1+\varepsilon m)^2}{(1+m)^2}(1+4\lambda)\bigg ],\\[1mm]
\Sigma_3=&\frac{[1+4\varepsilon(m+\varepsilon m^2+\lambda)](1+\varepsilon m)^2}{(1+m)^4}(1+4\varepsilon\lambda)(1+4\lambda)^2.
\end{align*}
Squaring both sides of \eqref{zeta} and rearranging terms we get the equation
\begin{align}
\zeta^2-\Sigma_1=\frac{2\sqrt{1+4\lambda}}{1+m}\bigg \{&\sqrt{1+4\varepsilon \lambda}[1+\varepsilon m-\sqrt{1+4\varepsilon(m+\varepsilon m^2+\lambda)}]\notag\\
&-\frac{1+\varepsilon m}{1+m}\sqrt{1+4\lambda}\sqrt{1+4\varepsilon(m+\varepsilon m^2+\lambda)}\bigg \}.
\label{zeta-1}
\end{align}
Squaring both sides of \eqref{zeta-1} and rearranging terms gives
\begin{align}
(\zeta^2-\Sigma_1)^2-4\Sigma_2=\frac{8\sqrt{1+4\varepsilon \lambda}(1+4\lambda)}{(1+m)^2}\bigg [&
\frac{[1+4\varepsilon(m+\varepsilon m^2+\lambda)](1+\varepsilon m)}{1+m}\sqrt{1+4\lambda}\notag\\
&-\frac{(1+\varepsilon m)^2}{1+m}\sqrt{1+4\varepsilon(m+\varepsilon m^2+\lambda)}\sqrt{1+4\lambda}\notag\\
&-(1+\varepsilon m)\sqrt{1+4\varepsilon\lambda}\sqrt{1+4\varepsilon(m+\varepsilon m^2+\lambda)}\bigg ].
\label{zeta-2}
\end{align}
Finally, squaring both sides of \eqref{zeta-2} and using \eqref{zeta-1}, we conclude that
$[(\zeta^2-\Sigma_1)^2-4\Sigma_2]^2-64\Sigma_3\zeta^2=0$ or, equivalently, $P_7(\lambda;m,\varepsilon)=0$, where
$P_7(\cdot;m,\varepsilon)$ is a seventh-order polynomial (see Appendix \ref{appendix-C} for the expression of the coefficients of the polynomial).

Finding the eigenvalues of $P_7(\cdot;m,\varepsilon)$ is quite challenging. The Routh-Hurwitz criterion (see, e.g., \cite[Chapter XV]{Gantmakher98}) gives relevant information on the eigenvalues without computing them explicitly, in particular whether the eigenvalues lie in the left halfplane ${\rm Re} \lambda <0$, by computing the Hurwitz determinants $\Delta_j$ ($j=1,\ldots,6$) associated with $P_7(\lambda;m,\varepsilon)$.
Unfortunately, our double-squaring method produces spurious eigenvalues which render Routh-Hurwitz criterion inefficient.
However, Orlando's formula (see \cite[Chapter XV, 7]{Gantmakher98}), a generalization of the well-known property for the sum of the roots of a quadratic equation, establishes a relation between the leading Hurwitz determinant $\Delta_{6}$ and the sums of all different pairs of roots of $P_7(\lambda;m,\varepsilon)$. In particular, $\Delta_{6}=0$ in the case when either $0$ is a double eigenvalue (i.e., $0$ is an eigenvalue with algebraic multiplicity two) or two eigenvalues are purely imaginary and conjugate.

The following one is the main result of this subsection.

\begin{proposition}
\label{give critical value}
There exist $\varepsilon_0>0$ and $\delta>0$, and a unique function $m^c: (0,\varepsilon_0)\to (6-\delta, 6+\delta)$ with $m^c(0)=6$, such that the polynomial $\widetilde{P}_7(\lambda;\varepsilon):=P_7(\lambda; m^c(\varepsilon), \varepsilon)$ has exactly one pair of purely imaginary roots $\pm i\omega(\varepsilon)$, with $\omega(\varepsilon)>0$.
Moreover, $\omega(\varepsilon)$ converges to $\sqrt{3}$ as $\varepsilon$ tends to $0$.
\end{proposition}

We first need a preliminary technical lemma:
\begin{lemma}\label{technical}
There exist $\upsilon_0>0$ and $\varepsilon_*>0$ such that, for all $m$ in the interval $[3,7]$ $($to fix ideas$)$, $\varepsilon\in (0,\varepsilon_*)$ and any
purely imaginary root $i\upsilon$ of $P_7(\cdot;m,\varepsilon)$, with $\upsilon>0$, it holds that $0<\upsilon<\upsilon_0$.
\end{lemma}

\begin{proof}
We observe that, if $i\upsilon$ is a root of $P_7(\cdot;m,\varepsilon)$, then, in particular, the imaginary part of $P_7(i\upsilon;m,\varepsilon)$, i.e., the
term $-a_0\upsilon^7+a_2\upsilon^5-a_4\upsilon^3+a_6\upsilon$ vanishes.

A straightforward computation (see Appendix \ref{appendix-C}) reveals that
\begin{align*}
\Im{P_7(i\zeta;m,\varepsilon)}=&-2048(\varepsilon-1)^4\varepsilon^2\zeta^7-8\varepsilon(m^2+3m+2)\zeta^5+O(\varepsilon^2)\zeta^5\\
&-128(2m^4-7m^2-3m-1)\zeta^3+O(\varepsilon)\zeta^3+a_6\zeta,
\end{align*}
for every $\zeta>0$, where we denote by $O(\varepsilon^k)$ terms depending only on $\varepsilon$ such that
the ratio $O(\varepsilon^k)/\varepsilon^k$ stays bounded and far away from zero for $\varepsilon$ in a neighborhood of zero.
Since $m^2+3m+2$ and $2m^4-7m^2-3m-1$ are both positive for $m\in [3,\infty)$, we can estimate
\begin{align*}
|\Im{P_7(i\zeta;m,\varepsilon)}|
\ge &[8(m^2+3m+2)-O(\varepsilon)]\varepsilon\zeta^5\!+\![128(2m^4-7m^2-3m-1)-O(\varepsilon)]\zeta^3\!-\!K|\zeta|,
\end{align*}
where $K:=\max\{|a_6(m,\varepsilon)|: m\in [3,7], \varepsilon\in (0,1]\}$. Hence, we can determine $\varepsilon_*>0$ such that
\begin{align}
|\Im{P_7(i\zeta;m,\varepsilon)}|
\ge & 64(2m^4-7m^2-3m-1)\zeta^3-K|\zeta|,\qquad\;\,m\in [3,7],\;\,\varepsilon\in (0,\varepsilon_*).
\label{imaginary}
\end{align}
The right-hand side of \eqref{imaginary} diverges to $\infty$ as $\zeta\to+\infty$. From this it follows that there exists $\upsilon_0>0$ such that
$|\Im{P_7(i\zeta;m,\varepsilon)}|>0$ for every $\zeta>\upsilon_0$ and this clearly implies that $\upsilon\le\upsilon_0$.
\end{proof}

\begin{proof}[Proof of Proposition \ref{give critical value}]
We split the proof into two steps.
\vskip 1mm
{\em Step 1}. First, we prove the existence of a function $m^c$ with the properties listed in the statement of the proposition.
For this purpose, we consider the sixth-order Hurwitz determinant ${\Delta}_6(m,\varepsilon)$ associated with the polynomial $P_7(\lambda;m,\varep)$.
It turns out that
${\Delta}_6(m,\varepsilon)=\varepsilon^2m^2C\widetilde\Delta_6(m,\varepsilon)$ for some positive constant $C$. As $\varepsilon\to 0$,
$\widetilde\Delta_6(\cdot,\varepsilon)$ converges to the function ${\Delta}_0$, which is defined by
\begin{align*}
{\Delta}_0(m)=&-m^{18}+8m^{17}+97m^{16}+42m^{15}-2129m^{14}-9376m^{13}-16811m^{12}\\
&-7866m^{11}+19913m^{10}+31292m^9-4309m^8-55466m^7-66363m^6\\
&-35480m^5-4729m^4+4666m^3+2628m^2+500m+24.
\end{align*}
Noticing that ${\Delta}_0(6)=0$ and $\frac{d}{dm}{\Delta}_0(6)>0$, it then follows from the Implicit Function Theorem that there exist $\varepsilon_0\in (0,\varepsilon_*)$, with $\varepsilon_*$ given by Lemma \ref{technical}, $\delta>0$ and a unique mapping  $m^c: (0,\varepsilon_0)\to (6-\delta, 6+\delta)$ with $m^c(0)=6$, such that $\widetilde\Delta_6(m^c(\varepsilon),\varepsilon)=0$ and $\frac{\partial}{\partial m}\widetilde\Delta_6(m^c(\varepsilon),\varepsilon)>0$ for $\varepsilon\in (0,\varepsilon_0)$. Then, upon an application of Orlando formula, it follows that either $0$ is a double root of $\widetilde P_7(\lambda;\varepsilon)$ or there exists at least one pair $\pm \omega(\varepsilon)i$ (with $\omega(\varepsilon)>0$) of purely imaginary roots of  $\widetilde P_7(\lambda;\varepsilon)$ for every $\varepsilon\in (0,\varepsilon_0)$. The first case is ruled out, since 0 is not a root of $\widetilde P_7(\lambda;\varepsilon)$. Indeed, $a_7(m,\varepsilon)$ converges to a positive limit as $\varepsilon$ tends to $0$.

\vskip 1mm
{\em Step 2}.
Next, we prove that $\pm\omega(\varepsilon)i$ is the unique pair of purely imaginary roots of the polynomial $\widetilde P_7(\lambda;\varepsilon)$ for every $\varepsilon\in (0,\varepsilon_0)$. For this purpose, we begin by observing that $\widetilde{P}_7(\cdot;\varepsilon)$ converges, locally uniformly in $\mathbb C$ as $\varepsilon\to 0$, to the fourth-order polynomial $\widetilde P_4$, defined by $\widetilde{P}_4(\lambda)=-6272(4\lambda+1)(\lambda-12)(\lambda^2+3)$ for every $\lambda\in\mathbb C$.
By Hurwitz Theorem, four roots of $\widetilde{P}_7(\lambda; \varep)$, say  $\lambda_1(\varep)$, $\lambda_2(\varep)$, $\lambda_3(\varep)$ and $\lambda_4(\varep)$ converge respectively to  $\lambda_1(0)=-\frac{1}{4}, \lambda_2(0)=12, \lambda_3(0)= \sqrt{3}i$ and $\lambda_4(0)=-\sqrt{3}i$. More precisely, for $r_1>0$ small enough, $\lambda_i(\varepsilon)$ ($i=1,\ldots,4$) is simple in the ball $B(\lambda_i(0),r_1)$ for $\varep\in (0,\varepsilon_0)$ (up to replacing $\varepsilon_0$ with a smaller value if needed).
Assume by contradiction  that there exists a positive infinitesimal sequence $\{\varepsilon_n\}$ such that, for any $n\in\N$, ($\lambda_{5}(\varepsilon_n),\lambda_6(\varepsilon_n)$) is another pair of purely imaginary and conjugate roots of $\widetilde P_7(\lambda; \varepsilon_n)$, different from $\pm\omega(\varepsilon_n)i$. By
Lemma \ref{technical}, $\nu(\varep_n)=|\lambda_5(\varepsilon_n)|\leq \upsilon_0$ for every $n\in\N$. Take a subsequence $\{\varep_{n_k}\}$ such that $\nu({\varep}_{n_k})$ converges as $k \to \infty$. The local uniform convergence in $\mathbb C$ of
$\widetilde P_7(\cdot;\varepsilon_n)$ to $\widetilde P_4$ implies that $\nu({\varep}_{n_k})$ tends to $\sqrt{3}$ as $k\to\infty$. Since the limit is independent of the choice of subsequence $\{\varep_{n_k}\}$, we conclude that $\nu(\varep_n)$ converges to $\sqrt{3}$ as $n\to\infty$.
Next, thanks to Hurwitz Theorem and the fact that $\lambda_3(\varep)$, $\lambda_4(\varep)$ converge to $\sqrt{3}i, -\sqrt{3}i$ respectively, the pair
($\lambda_5(\varepsilon_{n_k}),\lambda_6(\varepsilon_{n_k})$) coincides with ($\lambda_3(\varepsilon_{n_k}),\lambda_4(\varepsilon_{n_k})$) in $B(\sqrt{3}i,r_1)\times B(-\sqrt{3}i,r_1)$. This contradicts the fact that $\lambda_3(\varepsilon_{n_k}),\lambda_4(\varepsilon_{n_k})$ are both simple. Up to
replacing $\varepsilon_0$ with a smaller value if needed, we have proved that $(\omega(\varepsilon)i,-\omega(\varepsilon)i)$ is the unique pair of conjugate eigenvalues of
$\widetilde P_7(\cdot;\varepsilon)$ and  $\lambda_3(\varepsilon)=\omega(\varepsilon)i$ for every $\varepsilon\in (0,\varepsilon_0)$. The proof is now complete.
\end{proof}

\setcounter{tocdepth}{2}
\subsection{Hopf bifurcation theorem}
\label{subsect-5.2}
For fixed  $0<\varepsilon<\varepsilon_0$, $\varep_0$ and $\delta$ given by Proposition \ref{give critical value}, let us consider the fully nonlinear problem \eqref{FNLE},
where now we find it convenient to write $F(\pmb w;m)$ instead of $F(\pmb w)$ to make much more explicit the dependence of the nonlinear term $F$ on the bifurcation parameter $m$.
According to Proposition \ref{give critical value}, the bifurcation parameter $m$ has a critical value $m^c(\varep) \in (6-\delta,6+\delta)$. We intend to prove that a Hopf bifurcation occurs at $m=m^c(\varep)$ if $\varep$ is small enough. For $m$ close to $m^c(\varep)$, we are going to locally parameterize $m$ and $\pmb w$ by a parameter $\sigma \in (-\sigma_0,\sigma_0)$. To emphasize this dependence, we will write $\widetilde{m}(\sigma)$ and $\widetilde{\pmb w}(\cdot,\cdot;\sigma)$.

\begin{theorem}
\label{Hopf bifurcation theorem} For any fixed $\alpha\in (0,1)$, there exists $\tilde{\varep}_0\in (0,\varep_0)$, such that whenever $\varep\in (0,\tilde{\varep}_0)$ is fixed, the following properties are satisfied.
\begin{enumerate}[\rm (i)]
\item
There exist $\sigma_0>0$ and smooth functions $\widetilde{m}$, $\rho:(-\sigma_0,\sigma_0)\to\mathbb{R}$, $\widetilde{{\pmb w}}:(-\sigma_0,\sigma_0)\to C^{1+\alpha}(\R;\pmb{\mathcal W})\cap C^{\alpha}(\R;Q(D(L)))$, satisfying the conditions
$\widetilde{m}(0)=m^c$, $\rho(0)=1$ and $\widetilde{\pmb w}(\cdot,\cdot;0)$ $=0$.
In addition, $\widetilde{\pmb w}(\cdot,\cdot;\sigma)$ is not a constant if $\sigma\neq 0$, and $\widetilde{\pmb w}(\cdot,\cdot;\sigma)$ is a
$T(\sigma)$-periodic solution of the equation
\begin{eqnarray*}
\widetilde{\pmb w}_\tau(\cdot,\cdot;\sigma) = QL\widetilde{\pmb w}(\cdot,\cdot;\sigma) + F(\widetilde{\pmb w}(\cdot,\cdot;\sigma);\widetilde{m}(\sigma)), \qquad\;\, \tau \in \R,
\end{eqnarray*}
where $T(\sigma)=2\pi\rho(\sigma)\omega^{-1}$ and $\omega=\omega(\varepsilon)$ is defined in Proposition $\ref{give critical value}$.
\item
There exists $\eta_0$ such that if $\overline{m} \in (6-\delta_0, 6+\delta_0)$, $\bar{\rho}\in\R$ and $\pmb w \in C^{1+\alpha}(\mathbb{R};\pmb{\mathcal W})\cap C^{\alpha}(\mathbb{R};Q(D(L)))$ is a $2\pi\bar{\rho}\omega^{-1}$-periodic solution of the equation
$\overline{\pmb w}_\tau = QL\overline{\pmb w} + F(\overline{\pmb w};\overline{m})$ such that
\begin{equation*}
\|\overline{\pmb w}\|_{ C^{1+\alpha}(\R;\pmb{\mathcal W})}+\|\overline{\pmb w}\|_{C^{\alpha}(\R;Q(D(L)))}+|\bar{m}|+|1-\bar{\rho}|\leq\eta_0,
\end{equation*}
then there exist $\sigma\in(-\sigma_0,\sigma_0)$ and $\tau_0\in \R$ such that
$\overline{m}=\widetilde{m}(\sigma)$, $\bar\rho=\rho(\sigma)$ and $\overline{\pmb w}=\widetilde{\pmb w}(\cdot+\tau_0,\cdot;\sigma)$.    	
\end{enumerate}
\end{theorem}

\begin{proof}
We split the proof into two steps.
\vskip 1mm
\textsl{Step 1.} Here, we prove that there exists $\varepsilon_1>0$ such that $\pm \omega(\varepsilon)i$ are simple eigenvalues of $L$ (and, hence,
of the part of $L$ in $\pmb {\mathcal W}_Q=Q(\pmb {\mathcal W})$) for every $\varepsilon\in (0,\varepsilon_1]$ and there are no other eigenvalues on the imaginary axis, i.e., we prove that this operator satisfies the so-called resonance condition.

To begin with, let us prove that $\pm\omega(\varepsilon)i$ are eigenvalues of $L$. In view of Theorem \ref{thm-2.3}, we need to show that they are roots of the dispersion relation \eqref{dispersion epsilon}. For this purpose, we observe that the function $\widetilde{D}_\varepsilon:= D_\varepsilon(\cdot;m^c(\varepsilon))$ converges to  $\widetilde{D}_0$ locally uniformly in the strip $\{\lambda\in\mathbb{C}:|\Re\lambda|\le \ell\}$ (for $\ell$ small enough),
where
\begin{eqnarray*}
\displaystyle\widetilde D_0(\lambda)=-\lambda-\frac{1+\sqrt{1+4\lambda}}{2}+\frac{1}{14}[(13+2\lambda)\sqrt{1+4\lambda}+1+4\lambda],\qquad\;\,\lambda\in\mathbb C.
\end{eqnarray*}
The function $\widetilde D_0$ has just one pair of purely imaginary conjugate roots $\pm\sqrt{3}i$. Hurwitz theorem shows that there exists $r>0$ such that the ball $B(\sqrt{3}i,r)$ contains exactly one root $\lambda(\varepsilon)$ of $\widetilde D_{\varep}$ for each $\varepsilon$ small enough.
By the proof of Proposition \ref{give critical value}, we know that there exists $r_1>0$ such
that $\omega(\varepsilon)i$ is the unique root of $\widetilde P_7$ in the ball $B(\sqrt{3}i,r_1)$. Clearly, $\lambda(\varepsilon)$ is a root of the polynomial $\widetilde P_7$ and, Hurwitz theorem also shows that $\lambda(\varepsilon)$ converges to $\sqrt{3}i$ as $\varepsilon\to 0^+$. Therefore, for $\varepsilon$ small enough, both
$\lambda(\varepsilon)$ and $\omega(\varepsilon)i$ belong to $B(\sqrt{3}i,r_1)$ and, hence, they do coincide.
The same argument shows that $-\omega(\varepsilon)i$ is also a root of $\widetilde D_\varepsilon$. We have proved that there exists $\varepsilon_1\le\varepsilon_0$ such that  $\omega(\varepsilon)i$ and $-\omega(\varepsilon)i$ are both eigenvalues of $L$ of every $\varepsilon\in (0,\varepsilon_1]$. In particular, $\pm\omega(\varepsilon)i$ are simple roots of
the function $\widetilde D_{\varepsilon}$ and there are no other eigenvalues of $L$ on the imaginary axis.

To conclude that $\pm\omega(\varepsilon)i$ are simple eigenvalues of $L$ for each $\varepsilon\in (0,\varepsilon_1]$,
we just need to check that their geometric multiplicity is one. For this purpose, we observe that the proof of Theorem \ref{thm-2.3} shows that the eigenfunctions associated with the eigenvalues $\pm\omega(\varepsilon)i$ are given by
\begin{eqnarray*}
\begin{array}{lll}
\displaystyle u(\xi)=c_1e^{k_1\xi}+\frac{A}{H_{1,\lambda}}\bigg (\frac{e^{k_3\xi}}{k_3-k_2}-\frac{e^{k_3\xi}-e^{k_1\xi}}{k_3-k_1}\bigg )c_3,\quad &v(\xi)=c_3e^{k_3\xi}, &\xi<0,\\[3mm]
u(\xi)=c_6e^{k_2\xi}, &v(\xi)=c_8e^{k_6\xi}, &\xi\ge 0
\end{array}
\end{eqnarray*}
with $k_j=k_{j,\pm\omega(\varepsilon)i}$ and the constants $c_1$, $c_3$, $c_6$ and $c_8$ are determined through the equation \eqref{matrix} (with $\lambda=\pm\omega(\varepsilon)i$) where $F_1=\ldots=F_4=0$.
Since the rank of the matrix in \eqref{matrix} is three at $\lambda=\pm\omega(\varepsilon)i$, it follows at once that the geometric multiplicity of
$\pm\omega(\varepsilon)i$ is one.

\vskip 1mm
\textsl{Step 2:} Now, we check the nontransversality condition. We begin by observing that, for every $\varepsilon\in (0,\varepsilon_1]$, the function $D_{\varepsilon}$
is analytic with respect to $\lambda$ and continuously differentiable with respect to $m$ in $B(\sqrt{3}i,r)\times (6-\delta,6+\delta)$, where $r$ is such that
the ball $B(\sqrt{3}i,r)$ does not intersect the half line $(-\infty,-1/4]$.
We intend to apply the Implicit Function Theorem at $(\omega(\varep)i,m^c(\varep))$ for $\varep$ small enough.
In this respect, we need to show that the $\lambda$-partial derivative of $D_{\varepsilon}$ does not vanish at $(\lambda_3(\varep),m^c(\varep))$.
To this aim, we observe that
\begin{eqnarray*}
\lim_{\varepsilon\to 0^+}\frac{\partial D_{\varepsilon}}{\partial\lambda}(\omega(\varepsilon)i,m^c(\varepsilon))=\frac{\partial D_0}{\partial\lambda}(\sqrt{3}i,6)=\frac{5\sqrt{3}i-3}{49}.
\end{eqnarray*}
Therefore, there exists $\varepsilon_2\leq\varepsilon_1$ such that, if $\varepsilon\in (0,\varepsilon_2]$, the $\lambda$-partial derivative of $D_{\varepsilon}$ at $(\omega(\varepsilon)i,m^c(\varepsilon))$ does not vanish. Then, it follows from the Implicit Function Theorem that
for each $\varepsilon\in (0,\varepsilon_2]$, there exist $\delta_{\varepsilon}>0$, $r_{\varepsilon}<r$ and
a $C^1$-mapping $\lambda_{\varepsilon}:(m^c(\varepsilon)-\delta_{\varepsilon},m^c(\varepsilon)+\delta_{\varepsilon})\to B(\sqrt{3}i,r_{\varepsilon})$, such that $D_{\varepsilon}(\lambda_{\varep}(m),m)=0$ for all $m\in(m^c(\varepsilon)-\delta_{\varepsilon},m^c(\varepsilon)+\delta_{\varepsilon})$
and $\lambda_{\varepsilon}(6)=\omega(\varepsilon)i$.

As a consequence, there are two branches of conjugate isolated and simple eigenvalues, $\lambda_{\varep}(m)$ and $\overline{\lambda}_{\varep}(m)$, which cross the imaginary axis respectively at $\pm\omega(\varepsilon)i$ for $m=m^c(\varep)$.

It remains to determine the sign of the real part of the derivative of $\lambda_{\varep}$ at $m=m^c(\varep)$. Since
\begin{eqnarray*}
\lim_{\varepsilon\to 0^+}\frac{\partial\lambda_{\varepsilon}}{\partial m}(m^c(\varepsilon))
=-\bigg (\frac{\partial D_0}{\partial m}(\sqrt{3}i,6)\bigg )\bigg (\frac{\partial D_0}{\partial\lambda}(\sqrt{3}i,6)\bigg )^{-1}=\frac{3}{4}+\frac{\sqrt{3}}{12}i
\end{eqnarray*}
there exists $\varepsilon_3\le\varepsilon_2$ such that the real part of the derivative of $\lambda_{\varepsilon}$ is positive at $m^c(\varepsilon)$ for
any $\varepsilon\in (0,\varepsilon_3]$. which completes the proof of Step 2.

Applying \cite[Theorem 9.3.3]{Lunardi96}, the claims follow with $\tilde\varepsilon_0=\varepsilon_3$.
\end{proof}

\setcounter{tocdepth}{2}
\subsection{Bifurcation from the traveling wave}
As in Subsection \ref{stability theorem TW}, we rewrite the results in Theorem \ref{Hopf bifurcation theorem} in terms of problem \eqref{perturbation T}-\eqref{interface-Theta-Phi}. As above, $\varep$ is fixed in $(0,\tilde{\varep}_0)$; therefore, the traveling wave $\pmb U$ depends only on $m$, which itself is parameterized by $\sigma \in (-\sigma_0,\sigma_0)$. Accordingly, the traveling wave reads $\widetilde{\pmb U}(.;\sigma)$.

The following theorem expresses that there exists a bifurcated branch bifurcating from the traveling wave at the bifurcation point $m^c(\varep)$.
The proof can be obtained arguing as in the proof of Theorem \ref{stability theorem TW}. Hence, the details are skipped.

\begin{theorem}
For each $\sigma \in (-\sigma_0,\sigma_0)$,
the problem \eqref{perturbation T}-\eqref{interface-Theta-Phi} admit a non trivial solution $(\widetilde{{\pmb X}}(\cdot,\cdot;\sigma),\widetilde g(\cdot;\sigma))$ defined by:
\begin{align*}
&\widetilde{{\pmb X}}(\cdot,\cdot;\sigma)=\Theta_i^{-1}\widetilde w_1(\cdot,0;\sigma) {\widetilde{\pmb U}}'(\cdot;\sigma)+\widetilde{{\pmb w}}(\cdot,\cdot;\sigma)+ \widetilde{\pmb U}(\cdot;\sigma),\\
&\widetilde g(\tau;\sigma)=\tau+ \frac{\tau}{T(\sigma)}\int_0^{T(\sigma)}\frac{(L\widetilde{\pmb w}(r,\cdot;\sigma))_1(\sigma,0^+)}{\Theta_i-\widetilde w_1(r,0;\sigma)-\widetilde w'_1(r,0^+;\sigma)}dr + \widetilde{h}(\tau;\sigma),\quad\;\,\tau\in\mathbb R.
\end{align*}
where $\widetilde{{\pmb X}}(\cdot,\cdot;0)=\widetilde{\pmb U}(.;0)$, $\widetilde{\pmb w}$ is defined by Theorem $\ref{Hopf bifurcation theorem}$.
The function $\widetilde h(\cdot;\sigma)$ belongs to $C^{1+\alpha}(\mathbb R)$. Moreover, $\widetilde{\pmb X}(\cdot,\cdot;\sigma)$ and $\widetilde h(\cdot;\sigma)$ are periodic with period $T(\sigma)=2\pi\rho(\sigma)\omega^{-1}$. At the bifurcation point, the ``virtual period'' is $T(0)=2\pi\omega^{-1}$.
\end{theorem}

We refer to, e.g., \cite{NR97, Lorenzi04} for solutions which are periodic modulo a linear growth.

\section*{Acknowledgments} L.L. greatly acknowledges the School of Mathematical Sciences of the University of Science and Technology of China for the warm hospitality during his visit. M.M.Z. would like to thank the Department of Mathematical, Physical and Computer Sciences of the University of Parma for the warm hospitality during her visit. The authors wish to thank Peter Gordon, Congwen Liu and Gregory I. Sivashinsky for fruitful discussions.

\appendix

\section{General solution to the equation $\lambda \pmb u-{\mathcal L}\pmb u=\pmb f$}
\label{appendix-A}

Here, we collect the expression of the more general classical solution to the equation
$\lambda \pmb u-{\mathcal L}\pmb u=\pmb f$ when $\pmb f=(f_1,f_2)$ is a continuous function and $\lambda\in\C$.
We preliminarily note that, since $\Le>1$, the equation $k_{1,\lambda}=k_{4,\lambda}$ has no complex solutions $\lambda$.
The equation $k_{1,\lambda}=k_{3,\lambda}$ admits two complex conjugate solutions
\begin{equation}
\lambda_j^*=\frac{-A\Le+(-1)^ji\sqrt{A\Le(\Le-1)}}{\Le-1},\qquad\;\,j=1,2,
\label{lambda-k1-k3}
\end{equation}
whose real part is negative. Moreover, the equation $k_{2,\lambda}=k_{4,\lambda}$ admits no complex solutions.
Also the equation $k_{1,\lambda}=k_{4,\lambda}$ admits no solutions. Indeed, squaring twice the
equation $H_{1,\lambda}+H_{2,\lambda}=\Le-1$ we get $\lambda_1^*$ and $\lambda_2^*$ as solutions, which would imply that $k_{1,\lambda}=k_{2,\lambda}$. Obviously, this can not be the case.

Setting $\pmb u=(u,v)$, it turns out that, for any $\pmb f=(f_1,f_2)\in \bm{\mathcal W}$ and $\lambda\neq\{\lambda_1^*,\lambda_2^*\}$,
the general classical solution to the equation $\lambda \pmb u-{\mathcal L}\pmb u=\pmb f$ is given by
\begin{align}
u(\xi)=&\bigg (\!c_1\!-\!{\frac{1}{H_{1,\lambda}}\!\int_{0}^{\xi}\!(Av(s)\!+\!f_1(s))e^{-k_{1,\lambda}\!s}ds}\bigg )e^{k_{1,\lambda}\xi}\!+\!\bigg (\!c_2\!+\!{\frac{1}{H_{1,\lambda}}\!\int_{0}^{\xi}\!(Av(s)\!+\!f_1(s))e^{-k_{2,\lambda}\!s}ds}\bigg )e^{k_{2,\lambda}\xi}\notag\\
=&\bigg\{c_1  -\frac{A}{H_{1,\lambda}}\bigg [\frac{e^{(k_{3,\lambda}-k_{1,\lambda})\xi}-1}{k_{3,\lambda}-k_{1,\lambda}}c_3
+\frac{e^{(k_{4,\lambda}-k_{1,\lambda})\xi}-1}{k_{4,\lambda}-k_{1,\lambda}}c_4\bigg] \nonumber\\
&\;\,+\frac{A\Le}{H_{1,\lambda}H_{2,\lambda}}\bigg [\frac{e^{(k_{3,\lambda}-k_{1,\lambda})\xi}}{k_{3,\lambda}-k_{1,\lambda}}\int_0^\xi f_2(s)e^{-k_{3,\lambda}s}ds
-\frac{e^{(k_{4,\lambda}-k_{1,\lambda})\xi}}{k_{4,\lambda}-k_{1,\lambda}}\int_0^\xi f_2(s)e^{-k_{4,\lambda}s}ds\nonumber\\
&\phantom{\;\,\;\,+\frac{A\Le}{H_{1,\lambda}H_{2,\lambda}}\bigg [}+\frac{k_{3,\lambda}-k_{4,\lambda}}{(k_{4,\lambda}-k_{1,\lambda})(k_{3,\lambda}-k_{1,\lambda})}\int_0^\xi f_2(s)e^{-k_{1,\lambda}s}ds\bigg ]\notag\\
&\;\;-\frac{1}{H_{1,\lambda}}\int_0^{\xi} f_1(s)e^{-k_{1,\lambda}s}ds\bigg\} e^{k_{1,\lambda}\xi} \nonumber\\
&+\bigg\{c_2 +\frac{A}{H_{1,\lambda}}\bigg [\frac{e^{(k_{3,\lambda}-k_{2,\lambda})\xi}-1}{k_{3,\lambda}-k_{2,\lambda}}c_3
+\frac{e^{(k_{4,\lambda}-k_{2,\lambda})\xi}-1}{k_{4,\lambda}-k_{2,\lambda}}c_4\bigg ]\nonumber \\
&\;\;\;\;\;\,+\frac{A\Le}{H_{1,\lambda}H_{2,\lambda}}\bigg [ \frac{e^{(k_{4,\lambda}-k_{2,\lambda})\xi}}{k_{4,\lambda}-k_{2,\lambda}}\int_0^\xi f_2(s)e^{-k_{4,\lambda}s}ds
- \frac{e^{(k_{3,\lambda}-k_{2,\lambda})\xi}}{k_{3,\lambda}-k_{2,\lambda}}\int_0^\xi f_2(s)e^{-k_{3,\lambda}s}ds\nonumber\\
&\phantom{\;\;\;\;\;\,+\frac{A\Le}{H_{1,\lambda}H_{2,\lambda}}\bigg [\;\;}
-\frac{k_{3,\lambda}-k_{4,\lambda}}{(k_{3,\lambda}-k_{2,\lambda})(k_{4,\lambda}-k_{2,\lambda})}\int_0^\xi f_2(s)e^{-k_{2,\lambda}s}ds
\bigg ]\notag\\
&\;\;+\frac{1}{H_{1,\lambda}}\int_0^{\xi} f_1(s)e^{-k_{2,\lambda}s}ds
\bigg\} e^{k_{2,\lambda}\xi},
\label{resolvent-u}
\\
v(\xi)=&\bigg (c_3-\frac{\Le}{H_{2,\lambda}}\int_{0}^{\xi}f_2(s)e^{-k_{3,\lambda}s}ds\bigg )e^{k_{3,\lambda}\xi}+\bigg (c_4+\frac{\Le}{H_{2,\lambda}}\int_{0}^{\xi}f_2(s)e^{-k_{4,\lambda}s}ds\bigg )e^{k_{4,\lambda}\xi}
\label{resolvent-v}
\end{align}
for $\xi<0$ and
\begin{align}
\label{rs-u-positive}
u(\xi)&=\bigg (c_5-\frac{1}{H_{1,\lambda}}\int_{0}^{\xi}f_1(s)e^{-k_{1,\lambda}s}ds\bigg )e^{k_{1,\lambda}\xi}+\bigg (c_6+\frac{1}{H_{1,\lambda}}\int_{0}^{\xi}f_1(s)e^{-k_{2,\lambda}s}ds\bigg )e^{k_{2,\lambda}\xi},\\
\label{rs-v-positive}
v(\xi)&=\bigg (c_7-\frac{\Le}{H_{3,\lambda}} \int_{0}^{\xi}f_2(s)e^{-k_{5,\lambda}s}ds\bigg ) e^{k_{5,\lambda}\xi}+\bigg (c_8+\frac{\Le}{H_{3,\lambda}}\int_{0}^{\xi}f_2(s)e^{-k_{6,\lambda}s}ds\bigg )e^{k_{6,\lambda}\xi},
\end{align}
for $\xi\ge 0$. Here, $H_{i,\lambda}$ $(i=1,2,3)$ and $k_{j,\lambda}$ $(j=1,\ldots,6)$ are defined by \eqref{formula-1}-\eqref{formula-3}.

If $\lambda\in\{\lambda_1^*,\lambda_2^*\}$, then $k_{1,\lambda}=k_{3,\lambda}$. Hence,
in the definition of $u$ for $\xi<0$, the term
\begin{align*}
&-\frac{A(e^{k_{3,\lambda}-k_{1,\lambda}}-1)}{H_{1,\lambda}(k_{3,\lambda}-k_{1,\lambda})}c_3\notag\\
&+\frac{A\Le}{H_{1,\lambda}H_{2,\lambda}}\bigg [\frac{e^{(k_{3,\lambda}-k_{1,\lambda})\xi}}{k_{3,\lambda}-k_{1,\lambda}}\int_0^{\xi}f_2(s)e^{-k_{3,\lambda}s}ds+
\frac{k_{3,\lambda}-k_{4,\lambda}}{(k_{3,\lambda}-k_{1,\lambda})(k_{4,\lambda}-k_{1,\lambda})}
\int_0^{\xi}f_2(s)e^{-k_{1,\lambda}s}ds\bigg ]
\end{align*}
should be replaced by
\begin{align*}
-\frac{A}{H_{1,\lambda}}c_3\xi-\frac{A\Le}{H_{1,\lambda}H_{2,\lambda}}\int_0^{\xi}(s-\xi)f_2(s)e^{-k_{3,\lambda}s}ds-
\frac{A\Le\,e^{(k_{4,\lambda}-k_{1,\lambda})\xi}}{H_{1,\lambda}H_{2,\lambda}(k_{4,\lambda}-k_{1,\lambda})}\int_0^{\xi}f_2(s)e^{-k_{4,\lambda}s}ds.
\end{align*}

\section{On the equality $k_{1,\lambda}=k_{3,\lambda}$}
\label{appendix-B}
Here, we show that the solutions of the equation $k_{1,\lambda}=k_{3,\lambda}$, i.e., the complex numbers given by
\eqref{lambda-k1-k3}, are not solutions of the dispersion relation.
Since $(\Le^2+4\Le(A+\lambda_j^*))^{1/2}=\Le-1+(1+4\lambda^*_j)^{1/2}$, it is easy to see that $D(\lambda^*_j,\Theta_i,\Le)=0$ if and only if
\begin{align}
&\sqrt{\Le+4\lambda\Le}\,\bigg [1\pm 2i\frac{\sqrt{A\Le}}{\sqrt{\Le-1}}+(\Theta_i-1)\bigg (1-\frac{4A\Le}{\Le-1}\pm 4i\frac{\sqrt{A\Le}}{\sqrt{\Le-1}}\bigg )\bigg ]\notag\\
=&2A\Le-\bigg (\Le\pm 2i\frac{\sqrt{A\Le}}{\sqrt{\Le-1}}\bigg )\bigg [1\pm 2i\frac{\sqrt{A\Le}}{\sqrt{\Le-1}}+(\Theta_i-1)\bigg (1-\frac{4A\Le}{\Le-1}
\pm 4i\frac{\sqrt{A\Le}}{\sqrt{\Le-1}}\bigg )\bigg ].
\label{app-2}
\end{align}

Squaring both sides of \eqref{app-2} and identifying real and imaginary parts of the so obtained equation, after some long but straightforward computation
we get the following system for $\Le$ and $\Theta_i$:
\begin{equation}
\left\{
\begin{array}{l}
\displaystyle \Theta_i^2\!+\!A\Le\!+\!16(\Theta_i\!-\!1)^2\frac{A^2\Le^2}{(\Le\!-\!1)^2}\!-\!8\frac{A\Le}{\Le\!-\!1}\Theta_i(\Theta_i\!-\!1)(3\Theta_i\!-\!1)\!-\!\Theta_i\Le
\!+\!4(\Theta_i\!-\!1)\frac{A\Le^2}{\Le\!-\!1}=0,\\[4mm]
4A\Le(\Theta_i\!-\!1)(4\Theta_i\!-\!3)+(\Le\!-\!1)(3\Theta_i\!-\!4\Theta_i^2\!-\!\Le+2\Theta_i\Le)=0.
\end{array}
\right.
\label{app-3}
\end{equation}
First, we consider the second equation in \eqref{app-3}. Replacing $A$ with its value given by \eqref{eqn:A} and solving the so obtained equation with respect to $\Le$, we obtain that there are no positive solutions if $\Theta_i=1/2$ and, when $\Theta_i\in (0,1)\setminus\{1/2\}$,
then the equation has two real solutions
\begin{eqnarray*}
\Le_{\pm}=\frac{20\Theta_i^2-13\Theta_i-1\pm (400\Theta_i^4-552\Theta_i^3+169\Theta_i^2+14\Theta_i+1)^{\frac{1}{2}}}{2\Theta_i-1}.
\end{eqnarray*}
A straightforward computation reveals that $\Le_->1$ if and only if $\Theta_i\le 1/2$, whereas
$\Le_+>1$ if and only if $\Theta_i\in\left(\overline{\Theta}_i,1\right)$, where the value $\overline{\Theta}_i=(4+\sqrt{22})/12 \approx 0.724$ will play a significant role hereafter.

Now, we go back to the first equation in \eqref{app-3}. Replacing $A$ by its value, given by \eqref{eqn:A}, and taking $\Le=\Le_{\pm}$, we get the following equation
\begin{align}
&p(\Theta_i)=({\rm signum}(1-2\Theta_i))(1-\Theta_i)q(\Theta_i)\sqrt{400\Theta_i^4-552\Theta_i^3+169\Theta_i^2+14\Theta_i+1}
\label{app-6}
\end{align}
for $\Theta_i\in (0,1/2)\cup (\overline\Theta_i,1)$, where
\begin{align*}
p(\Theta_i)=&-38400\Theta_i^9+296896\Theta_i^8-800896\Theta_i^7+1041468\Theta_i^6-698658\Theta_i^5+218492\Theta_i^4-14718\Theta_i^3\notag\\
&-3894\Theta_i^2-298\Theta_i-8,\\[1mm]
q(\Theta_i)=&1920\Theta_i^6-11600\Theta_i^5+19164\Theta_i^4-12038\Theta_i^3+2174\Theta_i^2+251\Theta_i+8.
\end{align*}

It follows from the next lemma that \eqref{app-6} admits no solutions in the set $(0,1/2)\cup (\overline\Theta_i,1)$
and, consequently, the solutions of $k_{1,\lambda}=k_{3,\lambda}$ are not zeros of the dispersion relation.

\begin{lemma}
Function $q$ is positive in $(0,1/2)$ and negative in $(\overline{\Theta}_i,1)$. On the contrary, $p$ is negative in $(0,1/2)$ and positive in $(\overline{\Theta}_i,1)$.
\end{lemma}
\begin{proof}
Since the proof is easy but rather technical, we sketch it. In what follows, we denote by $c$ positive constants which may vary from line to line.
Similarly, by $p_k$ and $q_k$ we denote polynomials of degree $k$, which may vary from estimate to estimate.

We begin by considering the function $q$. For $\Theta_i\in (0,1/2)$, we can estimate
the sum of the first three terms in the definition of $q$ by $13364\Theta_i^4$, so that
$q(\Theta_i)>\Theta_i(13364\Theta_i^3-12038\Theta_i^2+2174\Theta_i+251)+8$ and the right-hand side of the previous inequality is not less than $-2\Theta_i+8$, so that $q$ is positive in $(0,1/2)$.

For $\Theta_i\in \left (\overline{\Theta}_i,1\right )$ things are a bit trickier.
Obviously, it suffices to prove that $q$ is negative in $(7/10,1)$. For this purpose, we observe that, since $q(7/10)<0$, we can estimate
$q<q-q(7/10)=:q_5$ in such an interval and
\begin{align}
q_5(\Theta_i)<&c\Theta_i(10\Theta_i-7)(120000\Theta_i^4-641000\Theta_i^3+749050\Theta_i^2-228040\Theta_i-23753)\notag\\
=& c\Theta_i(10\Theta_i-7)[(10\Theta_i-7)(24000\Theta_i^3-111400\Theta_i^3+71830\Theta_i+4673)-73975].
\label{q1}
\end{align}
Computing the maximum value of the above third-order polynomial in the interval $(7/10,1)$, we conclude that
$q(\Theta_i)<c\Theta_i(10\Theta_i-7)[17217(10\Theta_i-7)-73975]$, whose right-hand side is negative if $\Theta_i\in (7/10, \widehat{\Theta}_i)$, where $\widehat{\Theta}_i=67657/86085\approx 0.786$.
On the other hand, if $\Theta_i\in [\widehat{\Theta}_i,1)$, we can subtract from the fourth-order polynomial on the first line of \eqref{q1} its value
at $\widehat\Theta_i$ (which is negative) and, thus, estimate
$q(\Theta_i)\le c\Theta_i(10\Theta_i-7)(\Theta_i-\widehat\Theta_i)q_3(\Theta_i)$,
and $q_3$ is negative in the interval $[\widehat{\Theta}_i,1)$, as it is easily seen.
Thus, $q$ is negative in $\left (\overline{\Theta}_i,1\right )$ as claimed.

Next, we consider function $p$, first addressing the case when $\Theta_i\in (0,1/2)$. Note that
$p(\Theta_i)<p(\Theta_i)-p(0)=\Theta_ip_8(\Theta_i)<\Theta_i(p_8(\Theta_i)-p_8(1/2))=c\Theta_i(1-2\Theta_i)p_7(\Theta_i)$
for every $\Theta_i\in (0,1/2)$.
Iterating this procedure, in the end we deduce that
$p(\Theta_i)<c\Theta_i^3(1-2\Theta_i)^3p_4(\Theta_i)$ for each $\Theta_i\in (0,1/2)$.
Since $p_4(\Theta_i)<p_4(\Theta_i)-p_4(1/2)=c(1-2\Theta_i)p_3(\Theta_i)$ for every $\Theta_i\in (0,1/2)$ and $p_3$ is negative in $(0,1/2)$,
$p(\Theta_i)$ is negative for each $\Theta_i\in (0,1/2)$.

Let us now assume that $\Theta_i\in \left (\overline{\Theta}_i,1\right )$. Since $\overline{\Theta}_i>{18}/{25}=:\widetilde\Theta_i$, we can limit ourselves to proving that
$p$ is negative in $(\widetilde\Theta_i,1)$.
For this purpose, we observe that
\begin{align*}
p(\Theta_i)<&p(\Theta_i)-p(\widetilde\Theta_i)=c(\widetilde\Theta_i-\Theta_i)p_8(\Theta_i)
<c(\widetilde\Theta_i-\Theta_i)[p_8(\Theta_i)-p_8(\widetilde\Theta_i)]
=-c(\widetilde\Theta_i-\Theta_i)^2p_7(\Theta_i)\\
<& -c(\widetilde\Theta_i-\Theta_i)^2[p_7(\Theta_i)-p_7(1)]
= c(\widetilde\Theta_i-\Theta_i)^2(1-\Theta_i)p_6(\Theta_i).
\end{align*}
If $\Theta_i\in [0.745,0.75]$ then we estimate $\Theta_i^k\le 75\cdot 10^{-2k}$ for $k=4,6$,
$\Theta_i^k\ge 745\cdot 10^{-3k}$ for $k=1,2,3,5$, and conclude that $p_6$ and, hence, $p$ is negative in $[0.745,0.75]$. For $\Theta_i\in (0.75,1)$, we estimate
$p_6(\Theta_i)<p_6(\Theta_i)-p_6(3/4)=c(4\Theta_i-3)p_5(\Theta_i)$. Iterating this procedure, we conclude that
$p_6(\Theta_i)<(4\Theta_i-3)^3p_3(\Theta_i)$ and the polynomial $p_3$ is negative in $(0.75,1)$.
Finally, if $\Theta_i\in (0.72,0.745)$ then we set $\overline{\overline{\Theta}}_i=0.745$, estimate
\begin{align*}
p_6(\Theta_i)<&(p_6(\Theta_i)-p_6(\overline{\overline{\Theta}}_i))<c(\Theta_i-\overline{\overline{\Theta}}_i)p_5(\Theta_i)<c(\Theta_i-\overline{\overline{\Theta}}_i)
(p_5(\Theta_i)-p_5(\overline{\overline{\Theta}}_i))\\
=&c(\Theta_i-\overline{\overline{\Theta}}_i)^2p_4(\Theta_i)\le c(\Theta_i-\overline{\overline{\Theta}}_i)^2(p_4(\Theta_i)-p_4(\widetilde\Theta_i))=c(\Theta_i-\widetilde\Theta_i)(\Theta_i-\overline{\overline{\Theta}}_i)^2p_3(\Theta_i)
\end{align*}
and observe that $p_3$ is negative in $[0.72, 0.745)$. Thus, $p$ is negative in this interval as well.
Summing up, we have proved that $p$ is negative in $(\widetilde\Theta_i,1)$ as claimed. This concludes the proof.
\end{proof}

\section{The coefficients of the polynomial $P_7(\cdot;m,\varepsilon)$}
\label{appendix-C}
We collect here the expression of the coefficients $a_i=a_i(m,\varepsilon)$ ($i=0,1,\dots,7$) of the polynomial
$P_7(\lambda;m,\varepsilon)=a_0\lambda^7+a_1\lambda^6+a_2\lambda^5+a_3\lambda^4+a_4\lambda^3+a_5\lambda^2+a_6\lambda+a_7$,
which appears in Subsection \ref{p7}. They are given by
\begin{align*}
&a_0= 2^{11}(\varepsilon-1)^4\varepsilon^2;\\[2mm]
&a_1=-2^{11}(\varepsilon^2-\varepsilon)^{2}[(5\varepsilon^2-2\varepsilon+1)m^2+2(\varepsilon+1)^2m+4];\\[2mm]
&a_2=(\varepsilon^2-\varepsilon)\big [\varepsilon(59\varepsilon^3-9\varepsilon^2+17\varepsilon-3) m^4
+4\varepsilon(15\varepsilon^3+15\varepsilon^2+17\varepsilon+1)m^3\\
&\phantom{a_2:=(\varepsilon^2-\varepsilon)[\;}+4(\varepsilon+2)(\varepsilon^3+9\varepsilon^2+5\varepsilon+1)m^2\!-\!8(2\varepsilon^3-3\varepsilon^2-4\varepsilon-3)m-8(\varepsilon-1)(\varepsilon+2)\big ];\\[2mm]
&a_3\!=\! 2^7\big [\!-\varepsilon^2(5\varepsilon-1)(9\varepsilon^3+\varepsilon^2+7\varepsilon-1)m^{6}
-2\varepsilon^2(59\varepsilon^4-8\varepsilon^3+74\varepsilon^2+8\varepsilon-5)m^5\\
&\phantom{a_3= 2^7[\!}\!-\!4\varepsilon(4\varepsilon^5\!+\!27\varepsilon^4\!+\!24\varepsilon^3\!+\!37\varepsilon^2\!+\!6\varepsilon\!-\!2)m^{4}
\!+\!4\varepsilon(4\varepsilon^5\!+\!20\varepsilon^4\!-\!31\varepsilon^3\!-\!23\varepsilon^2\!-\!33\varepsilon\!-\!1)m^3\\
&\phantom{a_3= 2^7[\!}\!+\!4(9\varepsilon^{5}\!+\!27\varepsilon^4\!-\!15\varepsilon^3\!-\!24\varepsilon^2\!-\!12\varepsilon\!-\!1)m^2
\!+\!4(\varepsilon^4\!+\!17\varepsilon^3\!-\!7\varepsilon^2\!-\!9\varepsilon-2)m\!-\!4(\varepsilon\!-\!1)^2(2\varepsilon\!+\!1)\big ];\\[2mm]
&a_4=2^3\big [\varepsilon^2(9\varepsilon-1)^2(\varepsilon-1)^2m^{8}+8\varepsilon^2(\varepsilon-1)(45\varepsilon^3+5\varepsilon^2-21\varepsilon+3)m^{7}\\
&\phantom{a_4:=8[\,}\!+\!8\varepsilon(21\varepsilon^5\!-\!58\varepsilon^4\!-\!84\varepsilon^3\!-\!32\varepsilon^2\!+\!27\varepsilon-2)m^{6}
\!-\!16\varepsilon(34\varepsilon^5\!+\!42\varepsilon^4\!+\!113\varepsilon^3\!+\!81\varepsilon^2\!-\!7\varepsilon\!-\!7)m^{5}\\
&\phantom{a_4:=8[\,}\!-\!16(7\varepsilon^6\!+\!96\varepsilon^5\!+\!75\varepsilon^4\!+\!176\varepsilon^3\!+\!42\varepsilon^2\!-\!10\varepsilon\!-\!2)m^4
\!+\!16\varepsilon(6\varepsilon^4\!-\!75\varepsilon^3\!-\!65\varepsilon^2\!-\!117\varepsilon-5)m^3\\
&\phantom{a_4:=8[\,}
\!+\!16(29\varepsilon^4-7\varepsilon^3-48\varepsilon^2-31\varepsilon-7)m^2+32(\varepsilon-1)(7\varepsilon^2+14\varepsilon+3)m-16(\varepsilon-1)^2\big ];\\[2mm]
&a_5=2^5m\big [\varepsilon^2(\varepsilon-1)^2(9\varepsilon^2+\varepsilon-2)m^7+(\varepsilon^2-\varepsilon)(38\varepsilon^4+46\varepsilon^3-39\varepsilon^2+3)m^6\\
&\phantom{a_5=2^5m[}+(36\varepsilon^6+33\varepsilon^5-123\varepsilon^4-95\varepsilon^3+54\varepsilon^2-1)m^5\\
&\phantom{a_5=2^5m[}-(8\varepsilon^6-8\varepsilon^5+169\varepsilon^4+233\varepsilon^3+25\varepsilon^2-41\varepsilon-2)m^4\\
&\phantom{a_5=2^5m[}-(60\varepsilon^5+110\varepsilon^4+320\varepsilon^3+129\varepsilon^2-22\varepsilon-21)m^3\\
&\phantom{a_5=2^5m[}-4(16\varepsilon^4\!+\!37\varepsilon^3\!+\!41\varepsilon^2\!+\!7\varepsilon\!-\!5)m^2
\!+\!2(4\varepsilon^3\!-\!37\varepsilon^2\!-\!10\varepsilon\!-\!5)m\!+\!4(5\varepsilon^2\!-\!2\varepsilon\!-\!3)\big ];\\[2mm]
&a_6= 2^3m\big [2\varepsilon^2(\varepsilon-1)^2(\varepsilon+1)(2\varepsilon-1)m^7
+(\varepsilon^2-\varepsilon)(16\varepsilon^4+58\varepsilon^3-19\varepsilon^2-10\varepsilon+3)m^6\\
&\phantom{a_6= 2^3m[}+(16\varepsilon^6+72\varepsilon^5-39\varepsilon^4-171\varepsilon^3+42\varepsilon^2+17\varepsilon-1)m^5\\
&\phantom{a_6= 2^3m[}+2(20\varepsilon^5-12\varepsilon^4-113\varepsilon^3-69\varepsilon^2+41\varepsilon+5)m^4
-(2\varepsilon+1)(18\varepsilon^3+81\varepsilon^2+80\varepsilon-51)m^3\\
&\phantom{a_6= 2^3m[}
-4(28\varepsilon^3+31\varepsilon^2+20\varepsilon-15)m^2-4(11\varepsilon-3)(\varepsilon+1)m-8(1-\varepsilon)\big ];\\[2mm]
&a_7= 2^4(m^2+m^3)\big [\varepsilon^2(\varepsilon^2-1)(2\varepsilon-1)m^4+\varepsilon(2\varepsilon-1)(3\varepsilon^2-3\varepsilon-2)m^3\\
&\phantom{a_7= 2^4(m^2+m^3\;[\;}+(2\varepsilon^4-13\varepsilon^2+4\varepsilon+1)m^2-3(2\varepsilon-1)(\varepsilon+1)m+2(1-2\varepsilon)\big ].
\end{align*}

\end{document}